\documentclass{article}
\usepackage{amssymb,amsmath,amsfonts,amsthm}
\usepackage{color}
\usepackage{authblk}
\usepackage[dvipsnames]{xcolor}
\usepackage{graphicx}
\usepackage{fancyhdr} 
\usepackage{slashed}

\newtheorem{thm}{Theorem}[section]
\newtheorem{cor}[thm]{Corollary}
\newtheorem{lem}[thm]{Lemma}
\newtheorem{prop}[thm]{Proposition}
\newtheorem{prop-defn}[thm]{Proposition-Definition}

\theoremstyle{definition}

\newtheorem{rk}[thm]{Remark}

\newcommand{\C}{\mathbb C}

\newcommand{\Z}{\mathbb Z}

\newcommand{\N}{\mathbb N}

\newcommand\ind{{\rm 1\kern-.30em I}}
\def \bui#1#2{\mathrel{\mathop{\kern 0pt#1}\limits^{#2}}}

\newcommand{\biindice}[3]%
{%

\begin{array}[t]{c}
{\displaystyle #1}\\
{\scriptstyle #2}\\
{\scriptstyle #3}
\end{array}

}

\title{Eigenvalue estimates of the Kohn-Dirac operator on CR manifolds}

\author[1]{Georges Habib\thanks{\texttt{ghabib@ul.edu.lb}}}
\author[2]{Felipe Leitner\thanks{\texttt{felipe.leitner@uni-greifswald.de}}}

\affil[1]{\footnotesize Lebanese University, Faculty of Sciences II, Department of Mathematics, P.O. Box 90656 Fanar-Matn, Lebanon and Universit\'e de Lorraine, CNRS, IECL, 54506 Nancy, France}
\affil[2]{\footnotesize Institut f\"ur Mathematik und Informatik, Walther-Rathenau-Str. 47, 17489 Greifswald, Germany}
\date{}

\begin{document}
\begin{sloppypar}

\maketitle

\vspace*{-1cm}

\begin{abstract}
In this paper, we establish a new eigenvalue estimate for the Kohn-Dirac operator on a compact CR manifold.  The equality case of this estimate is characterized by the existence of a CR twistor spinor. We then classify CR manifolds carrying such spinors by showing that the Webster Ricci tensor has at most two eigenvalues. In this context, we construct several examples. 
\end{abstract}

\section{Introduction} 
Subject to our studies in this paper is the
so-called Kohn-Dirac operator on CR spin manifolds
of hypersurface type with
adapted pseudo-Hermitian structure (cf. \cite{Pe:05,S:17,L:18}).
Recall that a CR structure on a smooth manifold $M$ of odd dimension $2m+1\geq 3$ is given by a pair $(H(M),J)$ consisting
of a contact distribution $H(M)$, and some  complex structure $J$ on $H(M)$, subject to integrability conditions.
In case the contact distribution
is spin, the choice of an adapted contact form $\theta$ gives rise to a spinor derivative and a
corresponding Dirac operator $D_\theta$, known as Kohn-Dirac operator. We study eigenvalue estimates for $D_\theta$ on closed manifolds
and describe the geometric nature of the limiting case
for these estimates, which
is related to
some pseudo-Hermitian version of twistor equation. Examples of twistor spinors are provided.
Our studies are motivated by similar investigations for Riemannian and K{\"a}hlerian spin manifolds (cf. \cite{BFGK,Hij1,Kirch1,P:11}). In fact, 
the current setting of pseudo-Hermitian geometry
is closely related to the study of Dirac operators on
Riemannian foliations (cf. \cite{Habib1,Habib2}).\\

The general study of CR manifolds has a long
history and is subject to extensive research in the intersecting fields of complex analysis in several variables,
the theory of partial differential equations and differential geometry.
In complex analysis, CR manifolds of hypersurface type naturally occur as boundaries of strictly pseudoconvex domains in complex vector space.
Here, the tangential Cauchy-Riemannian equations become apparent, which are initiating the notion {\em CR analysis}.
It was already observed by J.J. Kohn in the 1960s that the natural sub-Laplacian on such boundaries admits a Hodge-like theory,
which relates harmonic forms to the cohomology of the tangential Cauchy-Riemann complex.
Kohn also introduced the embeddability problem for CR manifolds into complex vector space.
Further classical topics are CR extension problems and the equivalence problem for CR manifolds \cite{JJK1,JJK2,Bogg}.
The latter problem was solved by S.S. Chern and J. Moser 
\cite{CM} and reveals the geometric
nature of CR manifolds related to the theory of Cartan connections. Today, CR geometry is considered to be
a particular case of parabolic Cartan geometries
and allows the application of tractor calculus (see \cite{CG1,CS}). The framework of parabolic
geometries also explains the close relationship of CR geometry and conformal Lorentzian geometry. This is expressed by 
the famous Fefferman construction, which was originally discovered in the 1970's for the case
of boundaries of pseudoconvex domains. The Fefferman construction is a basic tool for
studying CR invariants (see \cite{Feff,Lee,CG1, ADS}).\\

Abstract CR manifolds and pseudo-Hermitian geometry were introduced in the works of  N. Tanaka \cite{Tanaka} and S. Webster \cite{Web}
in the late 1970s from a differential geometric point of view. Basic to their work is the discovery of a canonical  linear connection $\nabla^W$,
called Tanaka-Webster connection, which depends on the choice of the pseudo-Hermitian form.
In \cite{Tanaka}, this connection and related differential operators are studied
and, in particular, Bochner-type formulas for sub-Laplacian operators on differential forms are established.
Our approach to the geometric study of CR manifolds follows this spirit, extending investigations to spin geometric aspects.\\

It is important to notice that the choice of a pseudo-Hermitian $1$-form $\theta$ is not unique and, in fact,
any conformal change of $\theta$ by some positive function gives rise to (another) adapted contact form
for $H(M)$. The corresponding Kohn-Dirac operator
is not conformally covariant as a whole with respect to changing $\theta$. However, since
the structure group of the Tanaka-Webster connection is contained in the unitary group
$U(m)$, the CR spinor bundle $\Sigma(H(M))$ decomposes into $U(m)$-irreducible components, corresponding to the eigenvalues
of the Clifford action of $d\theta$. Only the Kohn-Dirac operator $D_\theta$ restricted to spinors
with $d\theta$-eigenvalue zero 
is conformally covariant (see \cite{L:21}).
The limiting case of the eigenvalue estimate for the Kohn-Dirac operator on such spinor fields
is characterized by a truely CR-invariant twistor equation.
We describe the geometry of CR twistor spinors in this middle dimension by a basic example in Section \ref{Examples}.
As to be expected, the Kohn-Dirac and twistor operators 
in this case are directly
linked to their counterparts on the corresponding Fefferman space.
For the other components of the  Kohn-Dirac operator 
this is not the case. In fact, these  components are not CR covariant. \\

This paper is organised as follows. In Section \ref{sec:review}, we review some general facts on CR manifolds. After introducing the Webster-Tanaka connection and its torsion, we consider the horizontal exterior differential which acts on the space of horizontal differential forms and show that its square does not necessarily vanish. We also show in Lemma \ref{lem:dwrtheta} that the horizontal exterior differential of the Ricci form is related to the Webster Ricci tensor of the characteristic vector field, meaning that, in contrast to K\"ahler manifolds, it is a closed form only when the torsion is divergence free. This formula is crucial in our computations. In Section \ref{sec:spincr}, we review the basic tools on CR spin manifolds and establish several identities for the  Kohn-Dirac operator that we need to use in our classification. The main result is established in Theorem \ref{thm:estimateeigenvalue} of Section \ref{sec:eigest}, where we get an estimate for the eigenvalues of the Kohn-Dirac operator on a CR manifold with positive Webster scalar curvature. To show the estimate, we use the CR twistor spinor equation (see \cite{G:93, Hij1, K:86} for the definition on K\"ahler manifolds) in order to bound the covariant derivative of any spinor field in terms of the Kohn-Dirac operator of that spinor. This technique allows to characterize the eigenvalue estimate by the means of CR twistor spinors. Sections \ref{sec:twisspigeneral} and  \ref{sec:twisspimiddle} are devoted to study the existence of CR twistor spinors in different components of the spinor bundle. Depending on the bundle where the CR twistor spinor is located, in Theorems \ref{thm:equalitycasegeneral} and \ref{eq:equalitycasemiddle} we show that the Webster Ricci tensor can have at most two eigenvalues. We end the paper by constructing two examples of CR manifolds carrying CR twistor spinors. The first one is a pseudo-Einstein case, meaning that the Webster Ricci tensor has one eigenvalue. 
The second one has a twistor in the middle dimension,
where $d\theta$ acts trivially, and 
two different eigenvalues for the Webster Ricci tensor. 
Notice that, it is shown in \cite{P:11} that K\"ahlerian twistor spinors in the middle dimension cannot occur on K\"ahler manifolds with constant positive scalar curvature. \\

{\bf Acknowledgment:}  This project was initiated when the first author was a fellow at the Alfried Krupp Wissenschaftskolleg in Greifswald which he would like to thank. Both authors would like also to thank the Alexander von Humboldt foundation, the University of Greifswald for the financial support and Ines Kath for her support. \\

\section{Review on CR manifolds}\label{sec:review}
In this section, we recall some basic facts on CR geometry. For more details, we refer to \cite{DT:06}.\\ 

Let $M$ be a connected manifold of odd dimension $2m+1$ for $m\geq 1$. A {\it CR structure of hypersurface type} on $M$ consists of a pair $(H(M),J)$ where $H(M)$ is a subbundle of $TM$ of corank $1$ and $J:H(M)\to H(M)$ is an almost complex structure. We say that $(H(M),J)$ is a {\it partially integrable CR structure} on $M$ if for any two sections $X,Y\in H(M)$, the vector field $[JX,Y]+[X,JY]$ is a section of $H(M)$ \cite{CS:00}. In this case, the bilinear pairing given by  
\begin{eqnarray*}
G:H(M)\times H(M)&\longrightarrow&TM/H(M),\\
(X,Y)&\longmapsto&-\overline{[X,JY]},
\end{eqnarray*}
is symmetric and satisfies $G(JX,JY)=G(X,Y)$. For a partially integrable CR structure, the Nijenhuis tensor is defined for $X,Y\in H(M)$ by
$$N_J(X,Y):=[X,Y]-[JX,JY]+J([JX,Y]+[X,JY]).$$
The CR structure is called {\it integrable} if the Nijenhuis tensor vanishes identically. Examples of integrable CR structures are provided by real hypersurfaces in $\C^{m+1}$, where $J$ is the standard complex structure on  $\C^{m+1}$ and $H(M):=(J\nu)^\perp$ with $\nu$ the unit normal vector field to $M$ in $\C^{m+1}$. An integrable CR structure $(M,H(M),J)$ is called a {\it strictly pseudoconvex CR manifold} if the distribution $H(M)$ is contact and $G$ is definite.  
If moreover $M$ is orientable, there exists a $1$-form $\theta$, whose kernel is the distribution $H(M)$, that is $H(M)={\rm ker}(\theta)$. In this case, the $2$-form $d\theta$ is nondegenerate and defines a unique vector field $T$ associated with $\theta$, called {\it characteristic vector field}, determined by 
$$T\lrcorner \theta=1 \quad\text{and}\quad T\lrcorner d\theta=0.$$
The tangent bundle $TM$ splits then as $TM=H(M)\oplus \mathbb{R}T$ and we denote by $\pi_\theta:TM\to H(M)$ the projection along $T$. For $X,Y\in H(M)$, we let  
$$g_\theta(X,Y):=\frac{1}{2} d\theta(X,JY),$$
 which defines a metric on $H(M)$. Since $(M^{2m+1}, H(M),J)$ is strictly pseudoconvex, the metric $g_\theta$ is either positive or negative definite. When $g_\theta$ is positive definite, we call $\theta$ a {\it pseudo-Hermitian structure}. In all the paper, we will consider an orientable, strictly pseudoconvex $(M^{2m+1}, H(M),J)$ endowed with a pseudo-Hermitian structure $\theta$. \\
 
For $X,Y,Z\in H(M)$, the Koszul formula,
 \begin{eqnarray*}
2g_\theta(\nabla^W_XY,Z)&=&X(g_\theta(Y,Z))+Y(g_\theta(X,Z))-Z(g_\theta(X,Y))+g_\theta(\pi_\theta[X,Y],Z)\\&&-g_\theta(\pi_\theta[X,Z],Y)-g_\theta(\pi_\theta[Y,Z],Y)
 \end{eqnarray*}
defines a partial connection
$$\nabla^W:\Gamma(H(M))\times \Gamma(H(M))\to \Gamma(H(M)),$$ which is metric with respect to $g_\theta$ and has no torsion in the sense that
$$\nabla^W_X Y-\nabla^W_Y X=\pi_\theta([X,Y]).$$
Extending the connection $\nabla^W$ in the $T$-direction by setting 
$$\nabla^W_T X=\frac{1}{2}([T,X]-J[T,JX]),$$
for $X\in H(M)$ and by $\nabla^W T=0$ allows to define a linear connection on $TM$, still denoted by $\nabla^W$, called the {\it Webster-Tanaka connection}. It is a metric connection with respect to the Webster metric $g^W:=g_\theta+\theta\circ\theta$ that has torsion. Indeed, the torsion ${\rm Tor}^W(X,Y):=\nabla^W_{X} Y-\nabla^W_{Y} X-[X,Y]$ is equal to  
$${\rm Tor}^W(X,Y)=d\theta(X,Y)T$$
and 
$${\rm Tor}^W(T,X)=-\frac{1}{2}([T,X]+J[T,JX])$$
for $X,Y\in H(M)$. The complex structure $J$ and the $2$-form $d\theta$ are both parallel with respect to the Webster-Tanaka connection. Hence $J$ becomes a K\"ahler structure on $H(M)$ with K\"ahler form equal to $\frac{1}{2}d\theta$. The Webster torsion is the endomorphism on $H(M)$ defined by $\tau(X):= {\rm Tor}^W(T,X)$, for $X\in H(M)$. One can easily show that $\tau$ is a symmetric tensor with vanishing trace. It also satisfies $\tau\circ J=-J\circ\tau$ and this composition is symmetric and trace free.

\begin{rk}
{\it It is not difficult to check that if $\widehat\nabla$ is the Levi-Civita connection of the metric $g^W$, then $\widehat\nabla_T T=0$ and $\widehat\nabla_X T=\tau(X)+JX$ hold for any $X\in H(M)$ (see \cite[Lem. 1.5.3]{S:17}). Therefore, if the torsion vanishes, the CR manifold is Sasakian. In this case the vector field $T$ defines a transversally K\"ahler foliation and the Tanaka-Webster connection on $H(M)$ coincides with the transversal Levi-Civita connection \cite[p. 48]{Tondeur}.}
\end{rk}

We denote by $R^W$ the curvature tensor associated with the connection $\nabla^W$, that is, 
$$R^W(X,Y):=\nabla^W_X\nabla^W_Y-\nabla^W_Y\nabla^W_X-\nabla^W_{[X,Y]}$$
and $$R^W(X,Y,Z,V):=g_\theta(R^W(X,Y)Z,V)$$
for all $X,Y,Z,V\in TM$, which is clearly skew-symmetric in the first two and the last two components. Since $\nabla^W$ has torsion, the first Bianchi identity states that 
$$\sum_{X,Y,Z\in TM}R^W(X,Y)Z=\sum_{X,Y,Z\in TM} ((\nabla_X^W {\rm Tor}^W)(Y,Z)+{\rm Tor}^W({\rm Tor}^W(X,Y),Z)).$$
Therefore, by setting $S$ the skew-symmetric endomorphism given by $S(X,Y):=(\nabla^W_X\tau)(Y)-(\nabla^W_Y\tau)(X)$ for $X,Y\in H(M)$, we get that $$R^W(T,X)Y-R^W(T,Y)X=S(Y,X).$$ A direct computation shows, by defining  $S(X,Y,Z):=g_\theta(S(X,Y),Z)$ for $X,Y,Z\in H(M)$, that 
\begin{equation}\label{eq:relationS}
S(X,Y,Z)=S(X,Z,Y)+S(Z,Y,X).
\end{equation}
Hence, by \eqref{eq:relationS}, we obtain
\begin{equation}\label{eq:curvatureT}
R^W(T,Z,Y,X)=S(Y,X,Z).
\end{equation}
Again, the first Bianchi identity applied to vectors in $H(M)$ along with the fact that the torsion in the $H(M)$-direction is $\nabla^W$-parallel allow to get that
\begin{equation}\label{eq:bianchiort}
\sum_{X,Y,Z\in H(M)}R^W(X,Y)Z=\sum_{X,Y,Z\in H(M)} d\theta(X,Y)\tau(Z).
\end{equation}
As a direct consequence, one finds
\begin{eqnarray}\label{eq:bianchiidentity2}
R^W(X,Y,Z,V)-R^W(Z,V,X,Y)&=&d\theta(Y,Z)\tau(X,V)+d\theta(X,V)\tau(Y,Z)\nonumber\\&&-d\theta(X,Z)\tau(Y,V)-d\theta(Y,V)\tau(X,Z)\nonumber\\
\end{eqnarray}
 for any $X,Y,Z,V\in H(M)$. The Webster Ricci endomorphism is the trace of the curvature operator associated to $\nabla^W$, that is, 
$${\rm Ric}^W(X):=\sum_{l=1}^{2m}R^W(X,e_l)e_l$$
for all $X\in TM$. Here $\{e_l\}_{l=1,\ldots,2m}$ is a local orthonormal basis of $H(M)$. Equation \eqref{eq:bianchiidentity2} as well as the properties of the endomorphism $\tau$ show that ${\rm Ric}^W$ is a symmetric tensor on $H(M)$. Also, with the help of \eqref{eq:curvatureT},  we have that 
\begin{equation}\label{eq:rict}
{\rm Ric}^W(T)=\sum_{l=1}^{2m}(\nabla^W_{e_l}\tau)e_l.
\end{equation}
Notice that ${\rm Ric}^W(T,T)=0$ and thus ${\rm Ric}^W(T)$ is in $H(M)$. Again with \eqref{eq:bianchiort}, we can prove that  
\begin{equation}\label{eq:riccrho}
{\rm Ric}^W(X,Y)=\rho_\theta(X,JY)+2(m-1)\tau(X,JY)
\end{equation}
for all $X,Y\in H(M)$, where $\rho_\theta$ is the Ricci $2$-form given by 
$$\rho_\theta(X,Y)=\frac{1}{2}\sum_{l=1}^{2m}g_\theta(R^W(Je_l,e_l)X,Y),$$
for all $X,Y\in H(M)$. Identity \eqref{eq:bianchiidentity2} allows to show that the Ricci form can be written as well as
\begin{equation}\label{eq:rhotheta}
\rho_\theta(X,Y)=\frac{1}{2}\sum_{l=1}^{2m}R^W(X,Y,Je_l,e_l).
\end{equation}
The Webster scalar curvature is the trace of the Webster Ricci curvature, that is ${\rm Scal}^W=\sum_{l=1}^{2m}g_\theta({\rm Ric}^W(e_l),e_l).$ Now,
the second Bianchi identity reads 
\begin{equation}\label{eq:secondbianchi}
\sum_{X,Y,Z\in TM}(\nabla^W_X R^W)(Y,Z)=-\sum_{X,Y,Z\in TM}R^W({\rm Tor}^W(X,Y),Z).
\end{equation}
Therefore, as in the Riemannian case, one can prove that the divergence of the Webster Ricci tensor is related to the Webster scalar curvature through the formulas 
 \begin{equation}\label{eq:divric}
 \sum_{l=1}^{2m}g_\theta((\nabla^W_{e_l}{\rm Ric}^W)e_l,X)=\frac{1}{2}X({\rm Scal}^W),\,\, \sum_{l=1}^{2m}g_\theta((\nabla^W_{e_l}{\rm Ric}^W)(T),e_l)=\frac{1}{2}T({\rm Scal}^W)
 \end{equation}
for all $X\in H(M)$. The second equation in \eqref{eq:divric} is obtained by using \eqref{eq:secondbianchi} as well as \eqref{eq:riccrho}. Using Equations \eqref{eq:rict}, \eqref{eq:riccrho} and \eqref{eq:divric}, the following equality  
\begin{equation}\label{eq:diverricciform}
\sum_{l=1}^{2m}(\nabla^W_{e_l}\rho_\theta)e_l=\frac{1}{2} J((d^W({\rm Scal}^W))^\sharp)-2(m-1){\rm Ric}^W(T)
\end{equation}
holds, where $d^W$ is the exterior differential in $H(M)$ given for any smooth function $f$ by $d^W(f):=\sum_{l=1}^{2m}e_l(f) e_l^*$. 
The $1$-form $\theta$ is called a pseudo-Einstein  structure \cite{L:88} on the CR manifold $(M,H(M),J)$ if the Ricci form is a multiple of $d\theta$, that is $\rho_\theta=\lambda d\theta$ for some smooth function $\lambda$. In this case, from equality \eqref{eq:riccrho}, we get that $\lambda=\frac{{\rm Scal}^W}{4m}$. Therefore, as a consequence of \eqref{eq:diverricciform}, we deduce that 
$${\rm Ric}^W(T)=\frac{1}{4m}J((d^W({\rm Scal}^W))^\sharp),$$
which means that the Webster scalar curvature is not necessarily constant. \\ 

Next, we will recall some facts about the exterior derivative. For some parts, we refer to \cite{S:17}. We define the space of horizontal forms as being the set of all differential forms on $M$ that are annihilated by the vector $T$, that means
$$\Omega^k(H(M))=\{\omega\in \Omega^k(M)|\,\, T\lrcorner\omega=0\}.$$
The horizontal exterior differential $d^W$ maps  $\Omega^k(H(M))$ to $\Omega^{k+1}(H(M))$ and is defined by 
\begin{equation}\label{eq:exterdiff}
(d^W\omega)(X_1,\ldots,X_{k+1}):=(d\omega)(X_1^H,\ldots,X_{k+1}^H),
\end{equation}
where $d$ is the exterior differential on $M$, $X_1,\ldots, X_{k+1}\in TM$ and $X^H:=\pi_\theta(X)$ is the projection to $H(M)$ of any vector field $X\in TM$. We denote by $\delta^W:\Omega^k(H(M))\to \Omega^{k-1}(H(M))$, the formal $L^2$-adjoint of $d^W$. It was shown in \cite[Lem. 3.1.4]{S:17} that $d^W$ and $\delta^W$ can be locally expressed as follows:
$$d^W=\sum_{l=1}^{2m}e_l^*\wedge(\nabla^W_{e_l})\circ {\rm pr}^H,\,\, \delta^W=-\sum_{l=1}^{2m}e_l\lrcorner\nabla^W_{e_l}.$$
In contrast to the exterior differential $d$, the square of $d^W$ is not necessarily equal to zero. Indeed, for any smooth function $f$ on $M$, we have $(d^W)^2 f=-T(f)d\theta$ and, thus, $d^W$ does not define a chain complex. In general, we have for $(d^W)^2$   the following: 

\begin{lem} \label{lem:squared}  On the space of horizontal differential forms, we have
$$(d^W)^2=-d\theta\wedge \tau^{[k]}-d\theta\wedge\nabla^W_T,$$
where $\tau^{[k]}$ is the canonical extension of the symmetric endomorphism $\tau$ given for any $\omega\in \Omega^k(H(M))$ by 
$$(\tau^{[k]}\omega)(X_1,\ldots,X_k)=\sum_{l=1}^k\omega(X_1,\ldots,\tau(X_l),\ldots,X_k),$$
for all $X_1,\ldots,X_k\in H(M)$.
\end{lem} 
\begin{proof} We choose a local orthonormal frame $\{e_l\}_{l=1,\ldots,2m}$ of $H(M)$ parallel at some point $x$ and  compute 
 \begin{eqnarray*}
(d^W)^2&=&\sum_{s,l=1}^{2m}e_s^*\wedge e_l^*\wedge\nabla^W_{e_s}\nabla^W_{e_l}\\
&=&\sum_{s,l=1}^{2m}e_s^*\wedge e_l^*\wedge R^W(e_s,e_l)-(d^W)^2-\sum_{s,l=1}^{2m}d\theta(e_s,e_l)e_s^*\wedge e_l^*\wedge\nabla^W_T.
 \end{eqnarray*}  
Hence, we get that $(d^W)^2=\frac{1}{2}\sum_{s,l=1}^{2m}e_s^*\wedge e_l^*\wedge R^W(e_s,e_l)-d\theta\wedge \nabla^W_T$. To compute the curvature term, we write it as $R^W(X,Y)=\sum_{l=1}^{2m} R^W(X,Y)e_l\wedge (e_l\lrcorner) $ and get
 \begin{eqnarray*}
 \sum_{s,l=1}^{2m}e_s^*\wedge e_l^*\wedge R^W(e_s,e_l)&=&-\sum_{s,l,n,p=1}^{2m}R^W(e_s,e_l,e_n,e_p)e_s^*\wedge e_l^*\wedge e_n^*\wedge (e_p\lrcorner)\\
 &\bui{=}{\eqref{eq:bianchiort}}&-2\sum_{s,l=1}^{2m}e_s^*\wedge e_l^*\wedge R^W(e_s,e_l)\\&&-6\sum_{s,l=1}^{2m}e_s^*\wedge (Je_s)^*\wedge\tau(e_l)^*\wedge 
 (e_l\lrcorner).
 \end{eqnarray*}
Hence, we deduce that $\sum_{s,l=1}^{2m}e_s^*\wedge e_l^*\wedge R^W(e_s,e_l)=-2d\theta\wedge\tau^{[k]}$ and the result follows. 
\end{proof}
In the following, we compute the exterior differential of the Ricci form $\rho_\theta$  in terms of ${\rm Ric}^W(T)$. This formula will be useful later in the computations. 
\begin{lem} \label{lem:dwrtheta} We have the formula 
\begin{equation}\label{eq:dwrhotheta}
d^W\rho_\theta=-(J{\rm Ric}^W(T))^\flat\wedge d\theta.
\end{equation}
\end{lem}
\begin{proof} For all $X,Y,Z\in H(M)$, we have
\begin{eqnarray*}
(d^W\rho_\theta)(X,Y,Z)&=&(\nabla^W_X\rho_\theta)(Y,Z)-(\nabla^W_Y\rho_\theta)(X,Z)+(\nabla^W_Z\rho_\theta)(X,Y).
\end{eqnarray*}
Using \eqref{eq:rhotheta}, we write $\rho_\theta=\frac{1}{2}{\rm tr}_{g_\theta}(R^W\circ J)$ and compute 
\begin{multline}\label{eq:drhotheta}
(d^W\rho_\theta)(X,Y,Z)\\
=\frac{1}{2}{\rm tr}_{g_\theta}(\left((\nabla^W_XR^W)(Y,Z)-(\nabla^W_YR^W)(X,Z)+(\nabla^W_ZR^W)(X,Y)\right)\circ J)\\
\bui{=}{\eqref{eq:secondbianchi}}-\frac{1}{2}{\rm tr}_{g_\theta}(\left(d\theta(X,Y)R^W(T,Z)+d\theta(Y,Z)R^W(T,X)+d\theta(Z,X)R^W(T,Y))\circ J\right)\\
=-\frac{1}{2}d\theta(X,Y)\sum_{l=1}^{2m}R^W(T,Z,Je_l,e_l)-\frac{1}{2}d\theta(Y,Z)\sum_{l=1}^{2m}R^W(T,X,Je_l,e_l)\\-\frac{1}{2}d\theta(Z,X)\sum_{l=1}^{2m}R^W(T,Y,Je_l,e_l).
\end{multline}
From \eqref{eq:curvatureT}, we  have
\begin{eqnarray}\label{eq:tracericciformulaT}
\sum_{l=1}^{2m}R^W(T,X,Je_l,e_l)&=&\sum_{l=1}^{2m}S(Je_l,e_l,X)\nonumber\\
&=&\sum_{l=1}^{2m}g_\theta((\nabla^W_{Je_l}\tau)e_l-(\nabla^W_{e_l}\tau)Je_l,X)\nonumber\\
&=&-2\sum_{l=1}^{2m}g_\theta((\nabla^W_{e_l}\tau)Je_l,X)\nonumber\\
&=&2g_\theta(J{\rm Ric}^W(T),X).\nonumber\\
\end{eqnarray}
Hence, after using that $d\theta(X,Y)=-2g_\theta(X,JY)$ for all $X,Y\in H(M)$, Equation \eqref{eq:drhotheta} reduces to 
\begin{eqnarray*}
(d^W\rho_\theta)(X,Y,Z)=2g_\theta(X,JY)g_\theta(J{\rm Ric}^W(T),Z)+2g_\theta(Y,JZ)g_\theta(J{\rm Ric}^W(T),X)\\+2g_\theta(Z,JX)g_\theta(J{\rm Ric}^W(T),Y),
\end{eqnarray*}
which is \eqref{eq:dwrhotheta}.  
\end{proof}

\begin{rk} {\it One can immediately see that the $2$-form $\rho_\theta$ is a $d^W$-closed form if and only if  ${\rm Ric}^W(T)$ vanishes, that is, the torsion is divergence free.}  
\end{rk}
\section{Spin structure on CR manifolds}\label{sec:spincr}
In this section, we are interested in  spin CR manifolds with pseudo-Hermitian form. We consider the Kohn-Dirac operator and get a new estimate for the eigenvalues of this operator. We also characterize the equality case by means of the so-called twistor equation. For some parts, we refer to \cite{L:18}.\\ 

In what follows, we consider a strictly pseudoconvex CR manifold $(M^{2m+1},H(M),J)$ with $m\geq 1$ and a fixed pseudo-Hermitian structure $\theta$. We assume that $H(M)$ carries a spin structure as a vector bundle. We denote by $\Sigma(H(M))$ its spinor bundle which is a complex vector bundle of complex rank $2^m$ and by $``\cdot"$ the Clifford multiplication  
\begin{eqnarray*} H(M)\otimes \Sigma(H(M))&\longrightarrow & \Sigma(H(M)),\\
	(X,\varphi) &\longmapsto & X\cdot\varphi,
\end{eqnarray*}
which satisfies $\langle X\cdot\varphi,\psi\rangle=-\langle\varphi,X\cdot\psi\rangle$, for any $\varphi,\psi\in \Sigma(H(M))$, where $\langle\cdot,\cdot\rangle$ is the Hermitian inner product on $\Sigma(H(M))$. As for spin manifolds, the Webster- Tanaka connection lifts to $\Sigma(H(M))$ and gives rise to a covariant derivative   $\nabla^W:\Gamma(TM)\otimes\Gamma(\Sigma(H(M)))\to \Gamma(\Sigma(H(M)))$, that we still denote by $\nabla^W$. This connection satisfies the rules 
$$\nabla^W_X(Y\cdot\varphi)=\nabla^W_X Y\cdot\varphi+Y\cdot\nabla^W_X\varphi,$$
and, 
$$X(\langle\varphi,\psi\rangle)=\langle\nabla^W_X\varphi,\psi\rangle+\langle\varphi,\nabla^W_X\psi\rangle,$$
for all $X\in \Gamma(TM), Y\in H(M)$ and $\varphi,\psi\in \Sigma(H(M))$. The curvature operator of the connection $\nabla^W$ is given, for any $X,Y\in \Gamma(TM)$ and $\varphi\in \Sigma(H(M))$, by
\begin{eqnarray*}
R^W(X,Y)\varphi&=&\nabla^W_X\nabla^W_Y\varphi-\nabla^W_Y\nabla^W_X\varphi-\nabla^W_{[X,Y]}\varphi\\
&=&\frac{1}{4}\sum_{k,l=1}^{2m}g_\theta(R^W(X,Y)e_k,e_l)e_k\cdot e_l\cdot\varphi.
\end{eqnarray*}
Here $\{e_l\}_{l=1,\ldots,2m}$ is a local orthonormal frame of $H(M)$. Contracting the curvature operator yields the Ricci identity (see \cite[Eq. 22]{L:18} for the proof), 
\begin{equation}\label{eq:riccifor}
\sum_{l=1}^{2m} e_l\cdot R^W(X,e_l)\varphi=\frac{1}{2}\rho_\theta(JX)\cdot\varphi+\frac{1}{2}\tau(X)\cdot d\theta\cdot\varphi-m\tau(JX)\cdot\varphi
\end{equation}
for all $X\in H(M)$ and $\varphi\in \Sigma(H(M))$. Clearly, when the torsion $\tau$ vanishes, the r.h.s. of the last identity reduces to the Webster Ricci tensor. However, despite the presence of the torsion term, tracing \eqref{eq:riccifor} yields
\begin{equation}\label{eq:tracericciformula}
\sum_{k,l=1}^{2m} e_k\cdot e_l\cdot R^W(e_k,e_l)\varphi=\frac{1}{2} {\rm Scal}^W\varphi. 
\end{equation}
This is due to the fact that $\tau$ and $\tau\circ J$ are both symmetric and trace-free endomorphisms. The Ricci identity in the $T$-direction is stated in the following result: 
\begin{prop} \label{Riccitdirection} (Ricci identity in $T$-direction) The identity
$$\sum_{l=1}^{2m}e_l\cdot R^W(T,e_l)\varphi=-\frac{1}{2}{\rm Ric}^W(T)\cdot\varphi$$ 
holds for any $\varphi\in \Sigma(H(M))$. 
\end{prop}
\begin{proof} Using the definition of the spinorial curvature operator, we compute for any $X\in H(M)$
\begin{eqnarray}\label{eq:ddeltatau}
R^W(T,X)\varphi&=&\frac{1}{4}\sum_{k,l=1}^{2m}R^W(T,X,e_k,e_l)e_k\cdot e_l\cdot\varphi\nonumber\\
&\bui{=}{\eqref{eq:curvatureT}}&\frac{1}{4}\sum_{k,l=1}^{2m}S(e_k,e_l,X)e_k\cdot e_l\cdot\varphi\nonumber\\
&=&\frac{1}{4}\sum_{k,l=1}^{2m}g_\theta((\nabla^W_{e_k}\tau)(e_l)-(\nabla^W_{e_l}\tau)(e_k),X)e_k\cdot e_l\cdot\varphi\nonumber\\
&=&\frac{1}{4}\sum_{k=1}^{2m}e_k\cdot (\nabla^W_{e_k}\tau)(X)\cdot\varphi-\frac{1}{4}\sum_{l=1}^{2m}(\nabla^W_{e_l}\tau)(X)\cdot e_l\cdot\varphi\nonumber\\
&=&\frac{1}{2}\sum_{k=1}^{2m}e_k\cdot (\nabla^W_{e_k}\tau)(X)\cdot\varphi+\frac{1}{2}g_\theta({\rm Ric}^W(T),X)\varphi.\nonumber\\
\end{eqnarray}
Now, contracting the last equation and using again that $X\cdot Y=X\wedge Y-g_\theta(X,Y)$, we write
\begin{eqnarray*}
\sum_{l=1}^{2m}e_l\cdot R^W(T,e_l)\varphi
&=&\frac{1}{2}\sum_{k,l=1}^{2m}e_l\cdot e_k\cdot (\nabla^W_{e_k}\tau)(e_l)\cdot\varphi+\frac{1}{2}{\rm Ric}^W(T)\cdot\varphi\\
&\bui{=}{\eqref{eq:rict}}&-\frac{1}{2}\sum_{k,l=1}^{2m}e_k\cdot e_l\cdot (\nabla^W_{e_k}\tau)(e_l)\cdot\varphi-\frac{1}{2}{\rm Ric}^W(T)\cdot\varphi.
\end{eqnarray*}
Since the first term in the r.h.s. is equal to $\frac{1}{2}d^W({\rm tr}(\tau))\cdot\varphi$ which clearly vanishes, this finishes the proof. 
\end{proof}
With respect to the complex decomposition 
$$H(M)\otimes\mathbb{C}=H^{1,0}(M)\oplus H^{0,1}(M)$$
associated with the eigenvalues $\pm i$ of the complex structure $J$, we define for any vector $X\in H(M)$
$$X^+=\frac{1}{2}(X-iJX)\in H^{1,0}(M),\,\, X^-=\frac{1}{2}(X+iJX)\in H^{0,1}(M).$$
 The K\"ahler $2$-form $\frac{1}{2}d\theta$ splits the spinor bundle as (see \cite{K:86, Hij1}) 
\begin{equation}\label{eq:decospin}
\Sigma(H(M))=\oplus_{r=0}^m\Sigma_r(H(M)),
\end{equation}
where $\Sigma_r(H(M))$ is the eigenspace associated with the eigenvalue $i(2r-m)$, that is, 
$$\Sigma_r(H(M))=\{\varphi\in \Sigma(H(M))|\,\, d\theta\cdot\varphi=2i(2r-m)\varphi\}.$$
The Clifford action of $X^+$ (resp. $X^-$) maps $\Sigma_r(H(M))$ to $\Sigma_{r+1}(H(M))$ (resp. $\Sigma_{r-1}(H(M))$). Also, there is a canonical $\mathbb{C}$- anti-linear parallel endomorphism $\mathfrak{j}:\Sigma(H(M))\to \Sigma(H(M))$ such that $\mathfrak{j}^2=(-1)^{\frac{m(m+1)}{2}}$ with $\mathfrak{j}:\Sigma_r(H(M))\to \Sigma_{m-r}(H(M))$ and $\mathfrak{j}(X\cdot\varphi)=\overline{X}\cdot\mathfrak{j}(\varphi)$(see \cite{K:86, Hij1}).\\ 

The Kohn-Dirac operator is a first order sub-elliptic differential operator acting on sections of the spinor bundle $\Sigma(H(M))$. It  was introduced in \cite{Pe:05} (see also \cite{S:17}) on contact metric structures and is defined by 
$$D_\theta=\sum_{l=1}^{2m} e_l\cdot\nabla^W_{e_l}.$$ 
It is a formally self-adjoint operator (if $M$ is closed) with respect to the $L^2$-inner product $(\cdot,\cdot)=\int_M\langle\cdot,\cdot\rangle {\rm vol}$, where ${\rm vol}$ is the volume form of the Webster metric $g^W$. With respect to the decomposition \eqref{eq:decospin}, the Kohn-Dirac operator splits as 
$D_\theta=D_++D_-$ where 
$D_\pm:\Gamma(\Sigma_r(H(M)))\to \Gamma(\Sigma_{r\pm 1}(H(M)))$ is given by $D_\pm=\sum_{l=1}^{2m} e_l^\pm\cdot\nabla^W_{e_l}$. Hence, one can easily see that $D_\theta:\Gamma(\Sigma_r(H(M)))\to \Gamma(\Sigma_{r-1}(H(M)))\oplus \Gamma(\Sigma_{r+1}(H(M)))$.
The operators $D_+$ and $D_-$ can be also expressed in terms of the operator 
$$D_\theta^c=\sum_{l=1}^{2m} Je_l\cdot\nabla^W_{e_l},$$ 
by $D_\pm=\frac{1}{2}(D_\theta\mp iD_\theta^c)$. An easy computation shows that 
$$(D_\theta^c)^2=D_\theta^2\quad\text{and}\quad D_\theta D_\theta^c+D_\theta^c D_\theta=0.$$ 
Hence, we deduce that 
$$D_+^2=D_-^2=0, \,\, D_+D_-+D_-D_+=D_\theta^2.$$ 
This shows in particular that the space $\Sigma_r(H(M))$ is preserved by $D_\theta^2$, meaning that $D_\theta^2:\Gamma(\Sigma_r(H(M)))\to \Gamma(\Sigma_r(H(M)))$.  Also, as the endomorphism $\mathfrak{j}$ is parallel, we have that $D_\pm\mathfrak{j}=\mathfrak{j}D_\mp$. In \cite[Eq. 9]{L:18}, it is shown that the Kohn-Dirac opeator has a Schr\"odinger-Lichnerowicz formula which involves the covariant derivative in $T$-direction. Namely, the identity is
$$D_\theta^2=\nabla^*\nabla+\frac{1}{4}{\rm Scal}^W-d\theta\cdot\nabla^W_T,$$
where $\nabla^*\nabla=-\sum_{l=1}^{2m}\nabla^W_{e_l}\nabla^W_{e_l}+\sum_{l=1}^{2m}\nabla^W_{\nabla^W_{e_l}e_l}$. The term $\nabla^W_T$ can be computed in terms of the Ricci form as follows \cite[p. 109]{L:18}
\begin{eqnarray}\label{eq:nabla^Wt}
\nabla^W_T&=&-\frac{1}{4m}\rho_\theta\cdot+\frac{1}{2m}\sum_{l=1}^{2m}\left(\nabla^W_{\nabla^W_{e_l}Je_l}-\nabla^W_{e_l}\nabla^W_{Je_l}\right).
\end{eqnarray}
 Hence after replacing, the Schr\"odinger-Lichnerowicz formula reduces to
 \begin{equation}\label{eq:schrLich}
 D_\theta^2=\nabla^*\nabla+\frac{1}{4}{\rm Scal}^W+\frac{1}{4m}d\theta\cdot\rho_\theta-\frac{1}{2m}\sum_{l=1}^{2m}d\theta\cdot\left(\nabla^W_{\nabla^W_{e_l}Je_l}-\nabla^W_{e_l}\nabla^W_{Je_l}\right)
 \end{equation}
 and, by restricting to $\Sigma_r(H(M))$, it takes the form \cite[Thm. 8.1]{L:18}
\begin{equation}\label{eq:Lich}
D_\theta^2=\frac{2r}{m}\nabla_{10}^*\nabla_{10}+\frac{2(m-r)}{m}\nabla_{01}^*\nabla_{01}+\frac{1}{4}{\rm Scal}^W+\frac{(2r-m)i}{2m}\rho_\theta\cdot
\end{equation}
where $\nabla_{10}^*\nabla_{10}=-2\sum_{l=1}^m\nabla^W_{e_l^-}\nabla^W_{e_l^+}$ and $\nabla_{01}^*\nabla_{01}=-2\sum_{l=1}^m\nabla^W_{e_l^+}\nabla^W_{e_l^-}$. Here, we choose an orthonormal frame parallel at some point with respect to the Tanaka-Webster connection. Notice that $\nabla^*\nabla=\nabla_{10}^*\nabla_{10}+\nabla_{01}^*\nabla_{01}$ as shown in \cite[p.109]{L:18}.

\begin{rk} {\it We point out from \eqref{eq:nabla^Wt} that the Clifford action of $\rho_\theta$ preserves the space $\Sigma_r(H(M))$, that is $\rho_\theta\cdot\Sigma_r(H(M))\subset \Sigma_r(H(M))$. }
\end{rk}

We finish this section by establishing some identities that we use several times in the next sections. We refer to \cite{P:11} for similar results on K\"ahler manifolds. 

\begin{lem} \label{commutatorformulas} Let $(M^{2m+1},H(M),J)$ be a strictly pseudoconvex CR manifold such that $H(M)$ is spin.  For any $X\in \Gamma(H(M))$ and $\rho\in \Omega^2(H(M))$, we have the following 
\begin{equation}\label{eq:comdy}
D_\theta(X\cdot)=\sum_{l=1}^{2m}e_l\cdot\nabla^W_{e_l}X\cdot-X\cdot D_\theta-2\nabla^W_X.
\end{equation}
\begin{equation}\label{eq:commomegad+}
[D_\theta,\rho]=(d^W\rho+\delta^W\rho)\cdot-2\sum_{l=1}^{2m}(e_l\lrcorner\rho)\cdot\nabla^W_{e_l}.
\end{equation}
\begin{equation}\label{eq:bracketnabla^WD}
[\nabla^W_X,D_\theta]=\frac{1}{2}\rho_\theta(JX)\cdot+\frac{1}{2}\tau(X)\cdot d\theta\cdot-m\tau(JX)\cdot-\sum_{l=1}^{2m}e_l\cdot\nabla^W_{\nabla^W_{e_l}X}-2JX\cdot\nabla^W_T.
\end{equation}
\begin{equation}\label{eq:bracketnabla^WT}
[\nabla^W_T,D_\theta]=-\frac{1}{2} {\rm Ric}^W(T)\cdot-\sum_{l=1}^{2m} \tau(e_l)\cdot\nabla^W_{e_l}.
\end{equation}
For a vector field $X\in H(M)$ parallel at some point, we have
\begin{equation}\label{eq:com}
D_\pm(X^\pm\cdot)=-X^\pm\cdot D_\pm,\, D_\pm(X^\mp\cdot)=-X^\mp\cdot D_\pm-2\nabla^W_{X^\mp},\,[D_\theta^2,X]=\nabla^*\nabla X\cdot-4JX\cdot\nabla^W_T.
\end{equation}
Also, the following identity 
\begin{eqnarray} \label{eq:bracketnabla^Wd2}
[\nabla^W_X,D_\theta^2]&=&JX\cdot {\rm Ric}^W(T)\cdot+\frac{1}{4}X({\rm Scal}^W)+J{\rm Ric}^W T\cdot X\cdot\nonumber\\&&+\frac{1}{2}\nabla^W_{JX}\rho_\theta\cdot-\nabla^W_{\rho_\theta(JX)}+d\theta\cdot R^W(T,X)-d\theta\cdot\nabla^W_{\tau(X)}\nonumber\\&&-2mR^W(T,JX)+2(m-2)\nabla^W_{\tau(JX)}-\nabla^W_{\nabla^*\nabla X}\nonumber\\&&-2\sum_{l=1}^{2m}R^W(X,e_l)\nabla^W_{e_l}+4\nabla^W_{JX}\nabla^W_T
\end{eqnarray}
holds. 
When restricting to $\Sigma_r(H(M))$ and taking $X$  to be parallel at some point, we have that 
 \begin{equation}\label{eq:bracketnabla^Wd+} 
[\nabla^W_{X^\pm},D_\pm]=\pm\frac{i}{2}\rho_\theta(X^\pm)\cdot\mp 2iX^\pm\cdot\nabla^W_T,
 \end{equation}  
 \begin{equation}\label{eq:bracketx+nabla^Wd+} 
[\nabla^W_{X^+},D_-]=2i(r-m)\tau(X^+)\cdot,\,\, [\nabla^W_{X^-},D_+]=2ir\tau(X^-)\cdot.
 \end{equation}
\end{lem}
\begin{proof} Equations \eqref{eq:comdy} and \eqref{eq:commomegad+} can be easily proven. In order to check Equation \eqref{eq:bracketnabla^WD}, we choose an orthonormal frame $\{e_l\}_{l=1,\ldots,2m}$ of $H(M)$ parallel at some point and compute 
\begin{eqnarray*} 
[\nabla^W_X,D_\theta]&=&\nabla^W_XD_\theta-D_\theta(\nabla^W_X)\\
&=&\sum_{l=1}^{2m}e_l\cdot\nabla^W_X\nabla^W_{e_l}-\sum_{l=1}^{2m}e_l\cdot\nabla^W_{e_l}\nabla^W_X\\
&=&\sum_{l=1}^{2m}e_l\cdot (R^W(X,e_l)-\nabla^W_{\nabla^W_{e_l}X}-d\theta(X,e_l)\nabla^W_{T})\\
&\bui{=}{\eqref{eq:riccifor}}&\frac{1}{2}\rho_\theta(JX)\cdot+\frac{1}{2}\tau(X)\cdot d\theta\cdot-m\tau(JX)\cdot-\sum_{i=1}^{2m}e_l\cdot\nabla^W_{\nabla^W_{e_l}X}\\&&-2JX\cdot\nabla^W_{T}.
\end{eqnarray*}
In the last equality, we use the fact that $X\lrcorner d\theta=2JX$. To prove \eqref{eq:bracketnabla^WT}, we make the same computations as before to get
\begin{eqnarray*} 
[\nabla^W_T,D_\theta]&=&\sum_{l=1}^{2m} e_l\cdot( R^W(T,e_l)+\nabla^W_{[T,e_l]})\\
&=&-\frac{1}{2} {\rm Ric}^W(T)\cdot-\sum_{l=1}^{2m} \tau(e_l)\cdot\nabla^W_{e_l},
\end{eqnarray*}
where in the last equality we use the Ricci identity in Proposition \ref{Riccitdirection} and the fact that $\tau$ is a symmetric endomorphism. Equations in \eqref{eq:com} can be easily deduced from \eqref{eq:comdy}. For \eqref{eq:bracketnabla^Wd2}, we use \eqref{eq:bracketnabla^WD} to write
\begin{eqnarray*} 
[\nabla^W_X,D_\theta^2]&=&[\nabla^W_X,D_\theta]D_\theta+D_\theta([\nabla^W_X,D_\theta])\\
&=&\frac{1}{2}\rho_\theta(JX)\cdot D_\theta+\frac{1}{2}\tau(X)\cdot d\theta\cdot D_\theta-m\tau(JX)\cdot D_\theta\\&&-2JX\cdot\nabla^W_T D_\theta
+D_\theta\Big(\frac{1}{2}\rho_\theta(JX)\cdot +\frac{1}{2}\tau(X)\cdot d\theta\cdot\\&&-m\tau(JX)\cdot -\sum_{l=1}^{2m}e_l\cdot\nabla^W_{\nabla^W_{e_l}X}-2JX\cdot\nabla^W_T\Big).
\end{eqnarray*}
Using \eqref{eq:commomegad+} for the $2$-form $d\theta$, we have that $D_\theta(d\theta \cdot)=d\theta\cdot D_\theta-4\sum_{l=1}^{2m}Je_l\cdot\nabla^W_{e_l}$, and the above commutator becomes equal to 
\begin{eqnarray*} 
[\nabla^W_X,D_\theta^2]
&=&-2JX\cdot[\nabla^W_T, D_\theta]+\frac{1}{2}\sum_{l=1}^{2m}e_l\cdot(\nabla^W_{e_l}\rho_\theta)(JX)\cdot-\nabla^W_{\rho_\theta(JX)}\\&&+\frac{1}{2}\sum_{l=1}^{2m}e_l\cdot(\nabla^W_{e_l}\tau)(X)\cdot d\theta\cdot+2\sum_{l=1}^{2m}\tau(X)\cdot Je_l\cdot\nabla^W_{e_l}\\&&-d\theta\cdot\nabla^W_{\tau(X)}-m\sum_{l=1}^{2m} e_l\cdot(\nabla^W_{e_l}\tau)(JX)\cdot+2m\nabla^W_{\tau(JX)}\\&&-\sum_{k,l=1}^{2m}e_k\cdot e_l\cdot\nabla^W_{\nabla^W_{e_k}\nabla^W_{e_l}X}+4\nabla^W_{JX}\nabla^W_T.
\end{eqnarray*}
To compute the different terms in the above equality, we use \eqref{eq:ddeltatau}, \eqref{eq:bracketnabla^WT} as well as 
\begin{multline*}
\sum_{l=1}^{2m} e_l\cdot(\nabla^W_{e_l}\rho_\theta)(JX)\cdot=\sum_{k,l=1}^{2m}(\nabla^W_{e_l}\rho_\theta)(JX,e_k)e_l\cdot e_k\cdot\\
=\sum_{l=1}^{2m}(\nabla^W_{e_l}\rho_\theta)(e_l,JX)-\sum_{l<k}((\nabla^W_{e_l}\rho_\theta)(e_k,JX)-(\nabla^W_{e_k}\rho_\theta)(e_l,JX))e_l\cdot e_k\cdot\\
\bui{=}{\eqref{eq:diverricciform}}\frac{1}{2}X({\rm Scal}^W)-2(m-1){\rm Ric}^W(T,JX)-\sum_{l<k}((d^W\rho_\theta)(e_l,e_k,JX)-(\nabla^W_{JX}\rho_\theta)(e_l,e_k))e_l\cdot e_k\cdot\\
\bui{=}{\eqref{eq:dwrhotheta}}\frac{1}{2}X({\rm Scal}^W)-2m{\rm Ric}^W(T,JX)+{\rm Ric}^W(T,X)d\theta\cdot+2J{\rm Ric}^W(T)\cdot X\cdot+(\nabla^W_{JX}\rho_\theta)\cdot
\end{multline*}
and, \\
\begin{eqnarray*}
\sum_{k,l=1}^{2m}e_k\cdot e_l\cdot\nabla^W_{\nabla^W_{e_k}\nabla^W_{e_l}X}&=&-\sum_{k=1}^{2m}\nabla^W_{\nabla^W_{e_k}\nabla^W_{e_k}X}+\frac{1}{2}\sum_{k,l=1}^{2m}e_k\cdot e_l\cdot\nabla^W_{R^W(e_k,e_l)X}\\
&=&\nabla^W_{\nabla^*\nabla X}+\frac{1}{2}\sum_{k,l,t=1}^{2m}R^W(e_k,e_l,X,e_t)e_k\cdot e_l\cdot\nabla^W_{e_t}\\
&\bui{=}{\eqref{eq:bianchiidentity2}}&\nabla^W_{\nabla^*\nabla X}+2\sum_{l=1}^{2m}R^W(X,e_l)\nabla^W_{e_l}+2\sum_{l=1}^{2m}JX\cdot\tau(e_l)\cdot\nabla^W_{e_l}\\&&+2\sum_{l=1}^{2m}\tau(X)\cdot Je_l\cdot\nabla^W_{e_l}+4\nabla^W_{\tau(JX)}
\end{eqnarray*}
to get that the commutator $[\nabla^W_X,D_\theta^2]$ becomes equal to
\begin{eqnarray*} 
[\nabla^W_X,D_\theta^2]&=&JX\cdot {\rm Ric}^W(T)\cdot+\frac{1}{4}X({\rm Scal}^W)+J{\rm Ric}^W T\cdot X\cdot\\&&+\frac{1}{2}\nabla^W_{JX}\rho_\theta\cdot-\nabla^W_{\rho_\theta(JX)}+d\theta\cdot R^W(T,X)-d\theta\cdot\nabla^W_{\tau(X)}\\&&-2mR^W(T,JX)+2(m-2)\nabla^W_{\tau(JX)}-\nabla^W_{\nabla^*\nabla X}\\&&-2\sum_{l=1}^{2m}R^W(X,e_l)\nabla^W_{e_l}+4\nabla^W_{JX}\nabla^W_T,
\end{eqnarray*}
which gives the desired result. 
Equations \eqref{eq:bracketnabla^Wd+} and \eqref{eq:bracketx+nabla^Wd+} can be easily deduced from \eqref{eq:bracketnabla^WD} when restricting to $\Sigma_r(H(M))$ and replacing $X$ by $X^\pm$ where we also use the fact that $\rho_\theta(X^\pm) = \rho_\theta(X)^\pm$ and $\tau(X^\pm)=\tau(X)^\mp$ which means $\rho_\theta(X^\pm)\cdot \Sigma_r(H(M))\subset \Sigma_{r\pm 1}(H(M))$ and $\tau(X^\pm)\cdot\Sigma_r(H(M))\subset
\Sigma_{r\mp 1}(H(M))$. This finishes the proof.
\end{proof}
\section{Eigenvalue estimates}\label{sec:eigest}
In \cite{L:18}, the author gives an estimate for the eigenvalues of the Kohn-Dirac operator, by assuming the non-negativity of the Webster Ricci tensor. The proof is based on introducing the CR twistor operators on $\Sigma_r(H(M))$ for each $r$ as an adaptation of the usual ones defined on K\"ahler manifolds \cite{G:93, Hij1, K:86}. In the following, we establish a new estimate where only a lower bound on the Webster scalar curvature is required. For this, we follow the computations done in \cite{P:11, L:18}. The CR twistor operator of type $r\in\{0,\ldots,m\}$ is defined as the projection of the covariant derivative onto the kernel of the Clifford multiplication in $\Sigma_r(H(M))$. That means, for any spinor $\varphi\in\Sigma_r(H(M))$, the operator is 
$$
T_r\varphi:=\nabla^W\varphi+\frac{1}{2(m-r+1)}\sum_{l=1}^{2m}e_l\otimes e_l^+\cdot D_{-}\varphi+\frac{1}{2(r+1)}\sum_{l=1}^{2m}e_l\otimes e_l^-\cdot D_{+}\varphi.$$
Hence, a spinor field $\varphi\in\Sigma_r(H(M))$ is called a CR twistor spinor if $T_r\varphi=0$, that is, 
\begin{equation}\label{eq:twistordef}
\nabla^W_X\varphi=-\frac{1}{2(m-r+1)}X^+\cdot D_{-}\varphi-\frac{1}{2(r+1)}X^-\cdot D_{+}\varphi,
\end{equation}
for all $X\in\Gamma(H(M))$. It is not difficult to check that if $\varphi$ is a CR twistor spinor in $\Sigma_r(H(M))$, then $\mathfrak{j}\varphi$ is also a CR twistor spinor in $\Sigma_{m-r}(H(M))$.  This is due to the fact that $\mathfrak{j}$ is a parallel endomorphism and satisfies $D_\pm\mathfrak{j}=\mathfrak{j}D_\mp$. Hence the space of CR twistor spinors is of even complex dimension when $r\neq\frac{m}{2}$. Now a direct computation as in \cite[Lem. 2.5]{P:11} shows that 
\begin{eqnarray}\label{eq:nabla^Wdplusminus}
|\nabla^W\varphi|^2&=& |T_r\varphi|^2+\frac{1}{2(r+1)}|D_+\varphi|^2+\frac{1}{2(m-r+1)}|D_-\varphi|^2\nonumber\\
&\geq&\frac{1}{2(r+1)}|D_+\varphi|^2+\frac{1}{2(m-r+1)}|D_-\varphi|^2
\end{eqnarray}
with equality if and only if $\varphi$ is a CR twistor spinor.  Recall that a spinor field $\varphi\in \Sigma_r(H(M))$ is called  holomorphic (resp. antiholomorphic) if $D_+\varphi=0$ (resp. $D_-\varphi=0)$. Now, we have

\begin{thm} \label{thm:estimateeigenvalue}
Let $(M^{2m+1}, H(M),J)$ be a closed strictly pseudoconvex CR manifold with positive Webster scalar curvature ${\rm Scal}^W$. Assume that $H(M)$ is spin. Then any eigenvalue $\lambda$ of $D_\theta^2$ restricted to $\Sigma_r(H(M))$  satisfies 
\begin{equation*}
\lambda \geq 
a_r \mathop{\rm inf}\limits_M\left(\frac{{\rm Scal}^W}{4}-\frac{\beta_{r,\rm max}^2}{2{\rm Scal}^W}-\frac{(2r-m)^2}{8m^2}{\rm Scal}^W\right) ,
\end{equation*}
where $a_r:=\frac{m\,\,{\rm max}(r+1,m-r+1)}{m\,\,{\rm max}(r+1,m-r+1)-{\rm min}(r,m-r)}$ for any $r\in \{0,\ldots,m\}$ and $\beta_{r,\rm max}$ denotes the biggest eigenvalue (in absolute value) of the symmetric endomorphism $i\rho_\theta\cdot$ restricted to $\Sigma_r(H(M))$ if $r\neq \frac{m}{2}$ and $\beta_{\frac{m}{2},\rm max}:=0$. If equality is realized, the corresponding eigenspinor $\varphi$ satisfies one of the following: 
\begin{enumerate} 
\item For $r=\frac{m}{2}: T_r\varphi=0$ and the Webster scalar curvature ${\rm Scal}^W$ is positive constant.
\item For $r>\frac{m}{2}$ (resp. $r<\frac{m}{2}$): $T_r\varphi=0, D_-\varphi=0$ (resp. $D_+\varphi=0$), $\rho_\theta\cdot\varphi=\frac{{\rm Scal}^W}{4m}d\theta\cdot\varphi$ and the Webster scalar curvature is constant.
\end{enumerate}
\end{thm}

\begin{proof} By applying \eqref{eq:Lich} to an eigenspinor $\varphi\in \Gamma(\Sigma_r(H(M))$ of $D_\theta^2$ associated with the eigenvalue $\lambda$, we get after taking the Hermitian product with $\varphi$ and integrating over $M$, the following
\begin{eqnarray}\label{eq:eigenvalueestimateproof}
\lambda\int_M|\varphi|^2{\rm vol} &=&\frac{2r}{m}\int_M|\nabla^W_{10}\varphi|^2{\rm vol}+\frac{2(m-r)}{m}\int_M|\nabla^W_{01}\varphi|^2{\rm vol}\nonumber\\&&+\frac{1}{4}\int_M{\rm Scal}^W|\varphi|^2{\rm vol}+\frac{(2r-m)i}{2m}\int_M\langle\rho_\theta\cdot\varphi,\varphi\rangle {\rm vol}\nonumber\\
&\geq &\frac{2 {\rm min}(r,m-r)}{m}\int_M|\nabla^W\varphi|^2 {\rm vol}+\frac{1}{4}\int_M{\rm Scal}^W|\varphi|^2{\rm vol}\nonumber\\&&
+\frac{(2r-m)i}{2m}\int_M\langle\rho_\theta\cdot\varphi,\varphi\rangle {\rm vol}\nonumber\\
&\bui{\geq}{\eqref{eq:nabla^Wdplusminus}} & 
\frac{{\rm min}(r,m-r)}{m\,\,{\rm max}(r+1,m-r+1)}\int_M\lambda|\varphi|^2 {\rm vol} +\frac{1}{4}\int_M{\rm Scal}^W|\varphi|^2{\rm vol}\nonumber\\&&+\frac{(2r-m)i}{2m}\int_M\langle\rho_\theta\cdot\varphi,\varphi\rangle {\rm vol}.
\end{eqnarray}
For $r=\frac{m}{2}$, we clearly get the required estimate. In this case, equality is characterized by the fact that the Webster scalar curvature is constant and the existence of a CR twistor spinor in $\Sigma_{\frac{m}{2}}(H(M))$. Assume now that $r\neq \frac{m}{2}$. We use the following pointwise estimate: 
$$\left|\frac{1}{\sqrt{{\rm Scal}^W}}\rho_\theta\cdot\varphi-\frac{\sqrt{{\rm Scal}^W}}{4m}d\theta\cdot\varphi\right|^2\geq 0,$$
to say that 
\begin{eqnarray*}
\frac{(2r-m)i}{m}\langle\rho_\theta\cdot\varphi,\varphi\rangle&\geq &-\frac{1}{{\rm Scal}^W} |\rho_\theta\cdot\varphi|^2-\frac{(2r-m)^2{\rm Scal}^W}{4m^2}|\varphi|^2\\
&\geq &-\frac{1}{{\rm Scal}^W}\beta_{r,\rm max}^2|\varphi|^2 -\frac{(2r-m)^2{\rm Scal}^W}{4m^2}|\varphi|^2.\\
\end{eqnarray*}
Therefore,  Inequality \eqref{eq:eigenvalueestimateproof} reduces to 
$$\frac{\lambda}{a_r}\int_M|\varphi|^2 {\rm vol}\geq\int_M\left(\frac{{\rm Scal}^W}{4}-\frac{\beta_{r,\rm max}^2}{2{\rm Scal}^W}-\frac{(2r-m)^2}{8m^2}{\rm Scal}^W\right)|\varphi|^2{\rm vol},$$
with $a_r:=\frac{m\,\,{\rm max}(r+1,m-r+1)}{m\,\,{\rm max}(r+1,m-r+1)-{\rm min}(r,m-r)}$, which gives the required estimate. If now, the equality is attained for $r>\frac{m}{2}$ (resp. $r<\frac{m}{2}$) , then we have equality in all the above inequalities. Therefore, $\varphi$ is an antiholomorphic (resp. holomorphic) CR twistor spinor which satisfies $\rho_\theta\cdot\varphi=\frac{{\rm Scal}^W}{4m}d\theta\cdot\varphi$ on the one hand and $|\rho_\theta\cdot\varphi|^2=\beta_{r,max}^2|\varphi|^2$ on the other hand. Thus, $\beta_{r,max}^2=\frac{({\rm Scal}^W)^2}{4m^2}(2r-m)^2$, and hence, 
$$\frac{\lambda}{a_r}=\frac{r(m-r)}{m^2}{\rm Scal}^W,$$ 
which means the Webster scalar curvature should be constant. 
\end{proof}

As a direct corollary, we obtain the following estimate for pseudo-Einstein CR manifolds. In this case, $\rho_\theta=\frac{{\rm Scal}^W}{4m}d\theta$ and, thus, $\beta_{r,max}=-\frac{{\rm Scal}^W}{2m}(2r-m)$. Hence, we get 
\begin{cor} Let $(M^{2m+1},H(M),J)$ be a closed strictly pseudoconvex CR manifold such that $H(M)$ is spin. Assume that $M$ is pseudo-Einstein, then we have the following estimate
\begin{equation*}
\lambda \geq \frac{r(m-r)a_r}{m^2}\mathop{\rm inf}\limits_M({\rm Scal}^W),
\end{equation*}
for any $r\in \{0,\ldots,m\}$.
\end{cor} 
When $r\geq \frac{m}{2}$, the above estimate coincides with the one in \cite[Thm. 11.2]{L:18}. In the following, we state the characterisation of the existence of CR twistor spinors in $\Sigma_{\frac{m}{2}}(H(M))$.
\begin{prop} \label{pro:middledimensioncha} Let $(M^{2m+1},H(M),J)$ be a strictly pseudoconvex CR manifold. Assume that $H(M)$ is spin and let $\varphi$ be a CR twistor spinor in $\Sigma_{\frac{m}{2}}(H(M))$. Then we have,
$D_\theta^2\varphi=\frac{m+2}{4(m+1)}{\rm Scal}^W\varphi$. If moreover the Webster scalar curvature ${\rm Scal}^W$ is positive constant and the manifold is closed, we can write
$$\varphi=\varphi_1+\varphi_2:=\frac{1}{\lambda} D_-D_+\varphi+\frac{1}{\lambda} D_+D_-\varphi$$
with $\lambda:=\frac{(m+2){\rm Scal}^W}{4(m+1)}$ and $\varphi_1$ (resp. $\varphi_2$) is an antiholomorphic (resp. holomorphic) CR twistor spinor in $\Sigma_{\frac{m}{2}}(H(M))$.  
\end{prop}

\begin{proof} Let $\varphi$ be a CR twistor spinor, then from \eqref{eq:twistordef}, we have that 
$$\nabla^W_X\varphi=-\frac{1}{m+2}X^+\cdot D_-\varphi-\frac{1}{m+2}X^-\cdot D_+\varphi$$
for all $X\in H(M)$. Hence, a direct computation shows that 
$$\nabla^*\nabla\varphi=\frac{1}{m+2}(D_+D_{-}\varphi+D_-D_{+}\varphi)=\frac{1}{m+2}D_\theta^2\varphi.$$
Therefore plugging this last identity in \eqref{eq:Lich} (for $r=\frac{m}{2}$) yields the identity $D_\theta^2\varphi=\frac{m+2}{4(m+1)}{\rm Scal}^W\varphi$ as stated. To prove the second part, we simply write
$$\varphi=\frac{1}{\lambda}D_\theta^2\varphi=\frac{1}{\lambda}D_-D_+\varphi+\frac{1}{\lambda}D_+D_-\varphi,$$
where we set $\lambda=\frac{m+2}{4(m+1)}{\rm Scal}^W\neq 0$. Hence, we get the decomposition of $\varphi$ into the antiholomorphic spinor $\varphi_1:=\frac{1}{\lambda}D_-D_+\varphi$ and the holomorphic spinor $\varphi_2:=\frac{1}{\lambda}D_+D_-\varphi$. Now, we can easily check that 
$$D_\theta^2\varphi_1=\frac{1}{\lambda}D_\theta^2D_-D_+\varphi=\frac{1}{\lambda}D_-D_+D_\theta^2\varphi=D_-D_+\varphi=\lambda\varphi_1.$$
The same computations can be done for $\varphi_2$. Hence, $\varphi_1$ (resp. $\varphi_2$) is an eigenspinor associated with the lowest eigenvalue estimate. Therefore by Theorem \ref{thm:estimateeigenvalue}, the spinor $\varphi_1$ (resp. $\varphi_2$) is a CR twistor spinor in $\Sigma_{\frac{m}{2}}(H(M))$ which is also antiholomorphic (resp. holomorphic). This finishes the proof.
\end{proof}

\section{CR twistor spinors in $\Sigma_r(H(M))$, when $r\neq\frac{m}{2}$}\label{sec:twisspigeneral}
In this section, we are going to characterize the equality case of the estimate in Theorem \ref{thm:estimateeigenvalue} for $r\neq \frac{m}{2}$. That means, we will study the existence of an antiholormorphic CR twistor spinor $\varphi$ satisfying $\rho_\theta\cdot\varphi=\frac{{\rm Scal}^W}{4m} d\theta\cdot\varphi$. We show the following
\begin{thm}\label{thm:equalitycasegeneral}
Let $(M^{2m+1},H(m),J)$ be a closed strictly pseudoconvex CR manifold of constant positive Webster scalar curvature ${\rm Scal}^W$. Assume there exists an antiholomorphic CR twistor spinor $\varphi$ of type $r$ with $0<r<m$ and $r\neq \frac{m}{2}$ such that $\rho_\theta\cdot\varphi=\frac{{\rm Scal}^W}{4m} d\theta\cdot\varphi$. Then the torsion vanishes and the manifold is strictly pseudo-Einstein. Moreover, the  spinor $\psi:=D_+\varphi$ is a holomorphic twistor spinor.
\end{thm}
\begin{proof} From \eqref{eq:twistordef}, a direct computation shows that 
$\nabla^*\nabla\varphi=\frac{1}{2(r+1)}D^2_\theta\varphi,$
where we have used that $D^2_\theta\varphi=D_-D_+\varphi$. Now, we take an orthonormal frame $\{e_l\}_{l=1,\ldots,2m}$ of $H(M)$ parallel at some point to compute
\begin{equation}\label{eq:nabla^Wnabla^Wj}
\sum_{l=1}^{2m}\nabla^W_{e_l}\nabla^W_{Je_l}\varphi=-\frac{1}{2(r+1)}\sum_{l=1}^{2m}(Je_l)^-\cdot \nabla^W_{e_l}D_{+}\varphi
=\frac{i}{2(r+1)}D^2_\theta\varphi.
\end{equation}
Therefore, inserting \eqref{eq:nabla^Wnabla^Wj} into Equation \eqref{eq:schrLich} and using that $\rho_\theta\cdot\varphi=\frac{{\rm Scal}^W}{4m} d\theta\cdot\varphi$,  allow to get 
\begin{equation}\label{eq:d2theravarphihol}
D_\theta^2\varphi=\frac{(m-r)(r+1){\rm Scal}^W}{(m+1)m}\varphi.
\end{equation}
By Equation \eqref{eq:nabla^Wt}, we have that 
\begin{equation}\label{eq:hol5'}
\nabla^W_T\varphi=\frac{i(m-2r-1){\rm Scal}^W}{8m(m+1)}\varphi.
\end{equation}
Now, we use Equations \eqref{eq:bracketnabla^Wd+} and \eqref{eq:bracketx+nabla^Wd+} to get from $D_-\varphi=0$, with the help of \eqref{eq:hol5'} the following two equations 
\begin{equation}\label{eq:taurhotheta}
\tau(X^+)\cdot\varphi=0,\,\,\, i\rho_\theta(X^-)\cdot\varphi=\frac{{\rm Scal}^W}{2m}X^-\cdot\varphi.
\end{equation}
In the last equation, we use \eqref{eq:com}. Also Equation \eqref{eq:bracketnabla^WT} gives that 
\begin{equation}\label{eq:rictvarphihol}
({\rm Ric}^W(T))^-\cdot\varphi=0.
\end{equation}
Now, from \eqref{eq:d2theravarphihol}, on one hand we deduce that $[\nabla^W_{X^+}, D^2_\theta]\varphi=0$, on the other hand, we obtain from  \eqref{eq:bracketnabla^Wd2} the following equation
\begin{eqnarray}\label{eq:identitytauric}
2i {\rm Ric}^W(T,X^+)\varphi+\frac{i}{2}\nabla^W_{X^+}\rho_\theta\cdot\varphi+4i(r-m)R^W(T,X^+)\varphi\nonumber\\+\frac{2i(r-m+1)}{r+1}\tau(X^+)\cdot D_+\varphi+\frac{1}{r+1}\sum_{l=1}^{2m}R^W(X^+,e_l)(e_l^-\cdot D_+\varphi)=0.
\end{eqnarray}
Here, we use \eqref{eq:rictvarphihol} and \eqref{eq:hol5'}. An easy computation allows to show that 
$R^W(T,X^+)\varphi=-\frac{1}{2(r+1)}\tau(X^+)\cdot D_+\varphi$ and that  
\begin{eqnarray*}
\sum_{l=1}^{2m} R^W(X^+,e_l)(e_l^-\cdot D_+\varphi)&=&({\rm Ric}^W(X^+))^-\cdot D_+\varphi+\sum_{l=1}^{2m}e_l^-\cdot R^W(X^+,e_l)D_+\varphi\\&=&
2ir\tau(X^+)\cdot D_+\varphi.
\end{eqnarray*}
where we have used \eqref{eq:riccrho},\eqref{eq:riccifor} to get the last equality. Hence, Identity \eqref{eq:identitytauric} becomes 
\begin{equation}\label{eq:tauric}
\tau(X^+)\cdot D_+\varphi=- {\rm Ric}^W(T,X^+)\varphi
\end{equation}
after taking into account that $\nabla^W_{X^+}\rho_\theta\cdot\varphi=0$, which comes from differentiating the equation $\rho_\theta\cdot\varphi=\frac{{\rm Scal}^W}{4m} d\theta\cdot\varphi$ in the $X^+$-direction. Differentiating  Equation \eqref{eq:tauric} and using the fact that $\nabla^W_{X^-}D_+\varphi=2i(r+1)\tau(X^-)\cdot\varphi$ which can be proven from the second equation in \eqref{eq:bracketx+nabla^Wd+} as well as \eqref{eq:com}, we get with the help of \eqref{eq:taurhotheta}
\begin{equation}\label{eq:rict-tau}
\frac{i(2r+1)}{r+1}({\rm Ric}^W(T))^-\cdot D_+\varphi+4(r+1)|\tau|^2\varphi=2i\delta^W(({\rm Ric}^W(T))^-)\varphi.
\end{equation}
Here $|\tau|^2:=\sum_{l=1}^{2m}g_\theta(\tau(e_l),\tau(e_l))$ is the operator norm of $\tau$. This equation will be useful later to deduce the result. Now, we will do the same technique as before by considering the commutator $[\nabla^W_{X^-}, D^2_\theta]\varphi$. First, we have 
\begin{eqnarray}\label{eq:comutatorx-I}
[\nabla^W_{X^-}, D^2_\theta]\varphi&\bui{=}{\eqref{eq:d2theravarphihol}}&-\frac{(m-r){\rm Scal}^W}{2(m+1)m}X^-\cdot D_+\varphi+\frac{1}{2(r+1)}D^2_\theta(X^-\cdot D_+\varphi)\nonumber\\
&\bui{=}{\eqref{eq:com}}&\frac{1}{2(r+1)}\nabla^*\nabla X^-\cdot D_+\varphi+\frac{2i}{r+1}X^-\cdot \nabla^W_T D_+\varphi\nonumber\\
&\bui{=}{\eqref{eq:bracketnabla^WT},\eqref{eq:hol5'}}&\frac{1}{2(r+1)}\nabla^*\nabla X^-\cdot D_+\varphi-\frac{(m-2r-1){\rm Scal}^W}{4m(m+1)(r+1)}X^-\cdot D_+\varphi\nonumber\\&&-\frac{i}{r+1}X^-\cdot({\rm Ric}^W (T))^+\cdot\varphi.
\end{eqnarray}
In the last equation, we have also used \eqref{eq:rictvarphihol}. On the other hand, we use \eqref{eq:d2theravarphihol} for the commutator to get that
\begin{eqnarray}\label{eq:comutatorx-II} 
[\nabla^W_{X^-},D_\theta^2]\varphi&=&-iX^-\cdot ({\rm Ric}^W (T))^+\cdot\varphi+i({\rm Ric}^W (T))^+\cdot X^-\cdot\varphi\nonumber-\frac{i}{2}\nabla^W_{X^-}\rho_\theta\cdot\varphi\nonumber\\&&-\frac{i}{2(r+1)}\rho_\theta(X^-)\cdot D_+\varphi+4ir R^W(T,X^-)\varphi\nonumber\\&&+\frac{1}{2(r+1)}\nabla^*\nabla X^-\cdot D_+\varphi+\frac{1}{r+1}\sum_{l=1}^{2m}R^W(X^-,e_l)(e_l^-\cdot D_+\varphi)\nonumber\\&&-\frac{(m-2r-1){\rm Scal}^W}{4m(m+1)(r+1)}X^-\cdot D_+\varphi.
\end{eqnarray}
An easy computation gives that 
$R^W(T,X^-)\varphi=\frac{1}{4(r+1)}X^-\cdot ({\rm Ric}^W (T))^+\cdot\varphi$ and that
\begin{eqnarray}\label{eq:rwx-d-}
\sum_{l=1}^{2m} R^W(X^-,e_l)(e_l^-\cdot D_+\varphi)&=&({\rm Ric}^W(X^-))^-\cdot D_+\varphi+\sum_{l=1}^{2m}e_l^-\cdot R^W(X^-,e_l)D_+\varphi\nonumber\\
&=&\frac{i}{2}\rho_\theta(X^-)\cdot D_+\varphi,
\end{eqnarray}
where we have used \eqref{eq:riccrho} and \eqref{eq:riccifor}. Hence comparing Equations \eqref{eq:comutatorx-I}  and \eqref{eq:comutatorx-II} yields, after plugging the curvature terms, the following
$$\nabla^W_{X^-}\rho_\theta\cdot\varphi=2({\rm Ric}^W (T))^+\cdot X^-\cdot\varphi.$$
Differentiating $\rho_\theta\cdot\varphi=\frac{{\rm Scal}^W}{4m} d\theta\cdot\varphi$ in the $X^-$-direction gives after comparing with the above equation that 
\begin{equation}\label{eq:rhothetad+}
-\frac{i}{4(r+1)}\rho_\theta\cdot X^-\cdot D_+\varphi=\frac{(2r-m){\rm Scal}^W}{8m(r+1)}X^-\cdot D_+\varphi-i ({\rm Ric}^W (T))^+\cdot X^-\cdot \varphi.
\end{equation}
Contracting \eqref{eq:rhothetad+} yields, after using $\sum_{l=1}^{2m}e_l\cdot\rho_\theta\cdot e_l^-=-2(r-1)\rho_\theta+i{\rm Scal}^W$ and $\sum_{l=1}^{2m}e_l^+\cdot e_l^-=-2r$ on $\Sigma_r(H(M))$ which can be proven by a straightforward computation, the following  
$$\frac{i}{2(r+1)}\rho_\theta\cdot D_+\varphi=-\frac{(2r+2-m){\rm Scal}^W}{4m(r+1)}D_+\varphi-2i ({\rm Ric}^W (T))^+\varphi.$$
Replacing this last equation again into \eqref{eq:rhothetad+} gives the following 
\begin{equation*}
-\frac{i}{2(r+1)}\rho_\theta(X^-)\cdot D_+\varphi=-\frac{{\rm Scal}^W}{4m(r+1)}X^-\cdot D_+\varphi+2i{\rm Ric}^W(T,X^-)\varphi.
\end{equation*}
Differentiating this last equation and contracting gives with the help of \eqref{eq:diverricciform} and \eqref{eq:bracketnabla^Wd+} after some lenghty computation 
\begin{eqnarray}\label{eq:rict-rhotheta}
\frac{i(m-1)}{r+1}({\rm Ric}^W(T))^-\cdot D_+\varphi-\frac{1}{4(r+1)}\left(|\rho_\theta|^2-\frac{({\rm Scal}^W)^2}{2m}\right)\varphi=\nonumber\\-2i\delta^W(({\rm Ric}^W (T))^+)\varphi.
\end{eqnarray}
Here, we also use the second equation in \eqref{eq:taurhotheta} and the fact that $\sum_{l=1}^{2m}\rho_\theta(e_l^-)\cdot e_l=-\rho_\theta+\frac{i}{2}{\rm Scal}^W$, which can be easily shown. The norm of $\rho_\theta$ in \eqref{eq:rict-rhotheta} is being defined as $|\rho_\theta|^2:=\sum_{l=1}^{2m} |\rho_\theta(e_l)|^2$.
Finally, comparing \eqref{eq:rict-tau} to \eqref{eq:rict-rhotheta} yields the following 
\begin{eqnarray}\label{eq:finalequation}
4(r+1)(m-1)|\tau|^2+\frac{2r+1}{4(r+1)}\left(|\rho_\theta|^2-\frac{({\rm Scal}^W)^2}{2m}\right)=\nonumber\\2i(m-1)\delta^W(({\rm Ric}^W (T))^-)+2i(2r+1)\delta^W(({\rm Ric}^W (T))^+).
\end{eqnarray}
Since we have pointwise that $0\leq|\rho_\theta-\frac{{\rm Scal}^W}{4m}d\theta|^2=|\rho_\theta|^2-\frac{({\rm Scal}^W)^2}{2m}$, which can be proven using \eqref{eq:riccrho}, the integral over $M$ of \eqref{eq:finalequation} yields the first statement of the theorem. Notice here that the divergence of a vector with respect to the Tanaka-Webster connection is equal to the divergence with respect to the Levi-Civita connection of the metric $g^W=d\theta^2\oplus g_\theta$. To prove the second part of the theorem, we clearly have that the spinor $D_+\varphi$ is holomorphic. Now, we use Equation \eqref{eq:bracketnabla^WD} and \eqref{eq:com} along with the fact that the manifold is pseudo-Einstein with vanishing torsion to get that 
$$\nabla^W_X(D_+\varphi)=\frac{1}{r+1}\nabla^W_{X^-}(D_+\varphi)-\frac{(r+1){\rm Scal}^W}{2m(m+1)}X^+\cdot \varphi\ .$$
Hence, we deduce that $\nabla^W_{X^+}(D_+\varphi)=-\frac{1}{2(m-r)}X^+\cdot D_-D_+\varphi$, where we use Equation \eqref{eq:d2theravarphihol} and $\nabla^W_{X^-}(D_+\varphi)=0$. Therefore,  $D_+\varphi$ is a CR twistor spinor. 
\end{proof}
\begin{rk} {\it The compactness of the manifold $M$ in the previous theorem can be dropped by just assuming ${\rm Ric}^W(T)=0$ meaning that the torsion $\tau$ is free divergence.}
\end{rk}
\section{CR twistor spinors in $\Sigma_{\frac{m}{2}}(H(M))$} \label{sec:twisspimiddle}
In this section, we study the CR twistor equation in $\Sigma_r(H(M))$ when $r=m/2$, for $m$ even. Recall that the Yamabe problem on closed CR manifolds states that one can always find an adapted pseudo-Hermitian structure $\tilde\theta=e^{2f}\theta$ such that the conformal Webster scalar curvature $\widetilde{\rm Scal}^W$ is constant \cite{JL}. Also, it is shown in \cite{L:18} that the CR twistor equation in the middle dimension is invariant under any adapted pseudo-Hermitian structure. Thus, we will be interested in the sequel at studying the twistor equation on CR manifolds with constant Webster scalar curvature ${\rm Scal}^W$. Also, it is pointed out in Proposition \ref{pro:middledimensioncha} that the existence of such a spinor $\varphi$ gives that $D_\theta^2\varphi=\frac{m+2}{4(m+1)}{\rm Scal}^W\varphi$ which after taking the Hermitian product with $\varphi$ and integrating over $M$  (the manifold $M$ is still assumed to be closed) implies that the Webster scalar curvature is nonnegative. When this latter is equal to zero, the CR twistor equation reduces to the parallelness of $\varphi$ and, when it is constant positive, the twistor spinor $\varphi$ can be always chosen to be antiholomorphic (resp. holomorphic) by Proposition \ref{pro:middledimensioncha}. Therefore, we will be interested at considering CR manifolds of constant Webster scalar curvature carrying an antiholomorphic CR twistor spinor.
\begin{prop} \label{thm: antoholomiddle} Let $(M^{2m+1},H(M),J)$ be a strictly pseudoconvex CR manifold  with $m$ even. Assume that the torsion vanishes and the Webster scalar curvature ${\rm Scal}^W$ is constant. Then, the Ricci form $\rho_\theta$ preserves the space of antiholomorphic CR twistor spinors in $\Sigma_{\frac{m}{2}}(H(M))$.
\end{prop}
\begin{proof} Let $\varphi$ be an antiholomorphic CR twistor spinor in $\Sigma_{\frac{m}{2}}(H(M))$. Recall that $D_\theta^2\varphi=\frac{m+2}{4(m+1)}{\rm Scal}^W\varphi$ as shown in Proposition \ref{pro:middledimensioncha} and Equation \eqref{eq:nabla^Wt} gives that 
\begin{equation}\label{eq:nablatmiddle}
\nabla^W_T\varphi=-\frac{1}{4m}\rho_\theta\cdot\varphi-\frac{i{\rm Scal}^W}{8m(m+1)}\varphi.
\end{equation}
Equation \eqref{eq:bracketx+nabla^Wd+} allows to get from the fact $D_-\varphi=0$, the following equation  
\begin{equation}\label{eq:scalwrhovarphi}
\frac{{\rm Scal}^W}{4m}X^-\cdot\varphi-\frac{i}{2}\rho_\theta(X^-)\cdot\varphi-\frac{i}{2m}X^-\cdot\rho_\theta\cdot\varphi=0.
\end{equation}
As the torsion vanishes, Equation \eqref{eq:bracketnabla^WT} gives with the help of \eqref{eq:nablatmiddle} that $D_-(\rho_\theta\cdot\varphi)=0$, which means that $\rho_\theta\cdot\varphi$ is antiholomorphic. Now it remains to show that $\rho_\theta\cdot\varphi$ is a CR twistor spinor. For this, we will prove that $\nabla^W_{X}\rho_\theta\cdot\varphi=0$. Since the torsion vanishes, the curvature $R^W(T,X^+)\varphi$ vanishes as well. Hence, using \eqref{eq:nablatmiddle} we deduce that $\nabla^W_{X^+}\rho_\theta\cdot\varphi=0$. Now, to show that $\nabla^W_{X^-}\rho_\theta\cdot\varphi=0$, we need to perform several computations. From one hand, the commutator $[\nabla^W_{X^-},D_\theta^2]$ is equal to 
\begin{eqnarray}\label{eq:nablax-middle}
[\nabla^W_{X^-},D_\theta^2]&=&\nabla^W_{X^-}D_\theta^2\varphi-D_\theta^2(\nabla^W_{X^-}\varphi)\nonumber\\
&\bui{=}{\eqref{eq:com}}&\frac{1}{m+2}\nabla^*\nabla X^-\cdot D_+\varphi+\frac{4i}{m+2}D_+(\nabla^W_T\varphi)\nonumber\\
&\bui{=}{\eqref{eq:nablatmiddle}}&\frac{1}{m+2}\nabla^*\nabla X^-\cdot D_+\varphi-\frac{i}{(m+2)^2}X^-\cdot\rho_\theta\cdot D_+\varphi\nonumber\\&&-\frac{{\rm Scal}^W}{2(m+1)(m+2)^2}X^-\cdot D_+\varphi.
\end{eqnarray}
In this computation, we have used the fact that
\begin{equation} \label{eq:d+rhothetavarphi}
D_+(\rho_\theta\cdot\varphi)=\frac{m}{m+2}\rho_\theta\cdot D_+\varphi-\frac{i{\rm Scal}^W}{m+2}D_+\varphi
\end{equation}
as a consequence of \eqref{eq:commomegad+}, since $\rho_\theta$ is closed and coclosed in view of \eqref{eq:diverricciform} and \eqref{eq:dwrhotheta}. From the other hand, we use Equation \eqref{eq:bracketnabla^Wd2} for the computation of the commutator to get after using \eqref{eq:rwx-d-} and comparing with \eqref{eq:nablax-middle}
\begin{eqnarray}\label{eq:nablawx-rhotheta}
-\frac{(m-2)i}{2m}\nabla^W_{X^-}\rho_\theta\cdot\varphi-\frac{2i}{m(m+2)^2}X^-\cdot\rho_\theta\cdot D_+\varphi-\frac{2i}{m(m+2)}\rho_\theta(X^-)\cdot D_+\varphi\nonumber\\+\frac{{\rm Scal}^W}{m(m+2)^2}X^-\cdot D_+\varphi=0.\nonumber\\
\end{eqnarray}
The curvature term $R^W(T,X^-)\varphi=0$, by the fact that the torsion vanishes. Therefore, by computing directly this term, we get after using \eqref{eq:nablatmiddle}, \eqref{eq:bracketnabla^WT} and \eqref{eq:d+rhothetavarphi}
\begin{eqnarray}\label{eq:rtx-middle}
R^W(T,X^-)\varphi&=&\nabla^W_T\nabla^W_{X^-}\varphi-\nabla^W_{X^-}\nabla^W_T\varphi-\nabla^W_{[T,X^-]}\varphi\nonumber\\
&=&-\frac{1}{2m(m+2)^2}X^-\cdot\rho_\theta\cdot D_+\varphi-\frac{i{\rm Scal}^W}{4m(m+2)^2}X^-\cdot D_+\varphi\nonumber\\&&+\frac{1}{4m}\nabla^W_{X^-}\rho_\theta\cdot\varphi-\frac{1}{2m(m+2)}\rho_\theta(X^-)\cdot D_+\varphi.\nonumber\\
\end{eqnarray}
Comparing \eqref{eq:nablawx-rhotheta} to \eqref{eq:rtx-middle} yields $\nabla^W_{X^-}\rho_\theta\cdot\varphi=0$. Finally, to show that $\rho_\theta\cdot\varphi$ is a CR twistor spinor, we write
\begin{eqnarray*} 
\nabla^W_{X}(\rho_\theta\cdot\varphi)&=&(\nabla^W_{X}\rho_\theta)\cdot\varphi+\rho_\theta\cdot \nabla^W_{X}\varphi\\
&=&-\frac{1}{m+2}\rho_\theta\cdot X^-\cdot D_+\varphi\\
&=&-\frac{1}{m+2}X^-\cdot\rho_\theta\cdot D_+\varphi-\frac{2}{m+2}\rho_\theta(X^-)\cdot D_+\varphi\\
&\bui{=}{\eqref{eq:nablawx-rhotheta}}&-\frac{m}{(m+2)^2}X^-\cdot\rho_\theta\cdot D_+\varphi+\frac{i{\rm Scal}^W}{(m+2)^2}X^-\cdot D_+\varphi\\
&\bui{=}{ \eqref{eq:d+rhothetavarphi}}&-\frac{1}{m+2} X^-\cdot D_+(\rho_\theta\cdot\varphi).
\end{eqnarray*}
This finishes the proof of the theorem.
\end{proof}

According to Proposition \ref{thm: antoholomiddle}, when $M$ has constant Webster scalar curvature ${\rm Scal}^W$ and vanishing torsion, the Clifford action of $\rho_\theta$ preserves the space of antiholomorphic CR twistor spinor in  $\Sigma_{\frac{m}{2}}(H(M))$. Since this action is skew-symmetric, one can always find an antiholomorphic CR twistor spinor which is also an eigenspinor of the symmetric endomorphsim $i\rho_\theta$. For this, we will restrict in the next theorem our study to the existence of an antiholomorphic CR twistor spinor $\varphi\in \Sigma_{\frac{m}{2}}(H(M))$ such that $\rho_\theta\cdot\varphi=ic\varphi$ for some real number $c$.

\begin{thm} \label{eq:equalitycasemiddle} Let $(M^{2m+1},H(M),J)$ be a strictly pseudoconvex CR manifold of constant Webster scalar curvature ${\rm Scal}^W$ and vanishing torsion with $m$ even. Let $\varphi$ be an antiholomorphic CR twistor spinor in $\Sigma_{\frac{m}{2}}(H(M))$ with $\rho_\theta\cdot\varphi=ic\varphi$ for some $c\in \mathbb{R}$. We set $K:=\frac{{\rm Scal}^W}{2m}+\frac{c}{m}$ and $\alpha:=-K+\frac{{\rm Scal}^W}{2(m+1)}$. Then, we have
\begin{itemize}
\item If $c=0$, the manifold is pseudo-Einstein. The spinor $\psi:=D_+\varphi$ is a holomorphic CR twistor spinor in $\Sigma_{\frac{m}{2}+1}(H(M))$. 
\item If $c=-\frac{(m+2){\rm Scal}^W}{4(m+1)}$, the manifold is Webster Ricci flat and the spinor field $\varphi$ is parallel if the manifold is compact. 
\item If $c\neq 0$ and $c\neq -\frac{(m+2){\rm Scal}^W}{4(m+1)}$. We let $r_0:=-\frac{2c}{\alpha-K}$. The Webster Ricci tensor satisfies
\begin{equation}\label{eq:tracericpower}
{\rm tr}(({\rm Ric}^W)^s)=(2m-r_0)K^s+r_0\alpha^s.
\end{equation}
for all $s\in \mathbb{N}$. In particular, it has two different eigenvalues $\alpha$ and $K$ with multiplicity $r_0$ and $2m-r_0$ respectively. 
\end{itemize}
\end{thm}
\begin{proof}
Since $\rho_\theta\cdot\varphi=ic\varphi$, then we get from Equation \eqref{eq:d+rhothetavarphi} that 
\begin{equation}\label{eq:rhothetad+middle}
\rho_\theta\cdot D_+\varphi=i\left(\frac{(m+2)c}{m}+\frac{{\rm Scal}^W}{m}\right) D_+\varphi.
\end{equation}
Hence, Equations \eqref{eq:scalwrhovarphi} and \eqref{eq:nablawx-rhotheta} become
\begin{equation}\label{eq:rhothetax-d+varphi}
i\rho_\theta(X^-)\cdot\varphi=K X^-\cdot\varphi,\,\, i\rho_\theta(X^-)\cdot D_+\varphi=K X^-\cdot D_+\varphi.
\end{equation}
Using \eqref{eq:bracketnabla^WD}, one can easily show that 
\begin{equation}\label{eq:nablax-d+middle}
\nabla^W_{X}D_+\varphi=\frac{i}{2}\rho_\theta(X^+)\cdot\varphi+\frac{\alpha}{2}X^+\cdot\varphi.
\end{equation}
Differentiating \eqref{eq:rhothetad+middle} and contracting gives after some lengthy computations with the help of \eqref{eq:diverricciform}, \eqref{eq:dwrhotheta}, \eqref{eq:rhothetax-d+varphi} and \eqref{eq:nablax-d+middle} the following
$$|\rho_\theta|^2=\frac{{\rm Scal}^W c(m+2)}{m(m+1)}+\frac{({\rm Scal}^W)^2}{2m}+\frac{2c^2}{m}.$$
Therefore, this last equation is equivalent to 
\begin{eqnarray}\label{eq:rhothetapseudoeinstein}
|\rho_\theta-\frac{{\rm Scal}^W}{4m}d\theta|^2&=&|\rho_\theta|^2-\frac{({\rm Scal}^W)^2}{2m}\nonumber\\
&=&\frac{{\rm Scal}^W c(m+2)}{m(m+1)}+\frac{2c^2}{m}.
\end{eqnarray}
Depending on the value of $c$, we discuss three cases: 
\begin{itemize} 
\item If $c=0$. In this case, Equation \eqref{eq:rhothetapseudoeinstein} gives that the manifold $M$ is pseudo-Einstein. Hence, using \eqref{eq:nablax-d+middle}, we get
\begin{eqnarray*}
\nabla^W_{X}(D_+\varphi)
=-\frac{1}{m}X^+\cdot D_-D_+\varphi.
\end{eqnarray*}
In the last equality, we use that $D_\theta^2\varphi=D_-D_+\varphi=\frac{(m+2){\rm Scal}^W}{4(m+1)}\varphi$. This shows $D_+\varphi$ is a holomorphic twistor spinor in $\Sigma_{\frac{m}{2}+1}(H(M))$.  
\item If $c=-\frac{(m+2){\rm Scal}^W}{4(m+1)}$.  We use again  \eqref{eq:rhothetapseudoeinstein} to get the following identity
$$|\rho_\theta-\frac{{\rm Scal}^W}{4m}d\theta|^2=-\frac{(m+2)^2({\rm Scal}^W)^2}{8m(m+1)^2}\leq 0.$$
Hence, the Webster scalar curvature must vanish as well as $\rho_\theta$, which means the manifold is Webster Ricci flat. Therefore $D_\theta^2\varphi=0$ and, thus, $D_\theta\varphi=0$ if $M$ is compact. The spinor $\varphi$ becomes then parallel. This shows the second part. 
\item If $c\neq 0$ and $c\neq -\frac{(m+2){\rm Scal}^W}{4(m+1)}$. In this case, we have $\alpha\neq K$ and we let $r_0=-\frac{2c}{\alpha-K}$. Next, we show that the Webster Ricci tensor has the eigenvalues $\alpha$ and $K$ with mutiplicity $r_0$ and $2m-r_0$ respectively. For this, we proceed as in \cite{M:97}. We define the skew-symmetric endomorphism $\rho_{s,\theta}(X):=({\rm Ric}^W)^s(JX)$ on $H(M)$ and denote the following assertions by: 
\begin{gather*} 
(a_s): {\rm tr}(({\rm Ric}^W)^s)=(2m-r_0)K^s+r_0\alpha^s\\
(b_s): \nabla^W_X\rho_{s,\theta}\cdot\varphi=0\\
(c_s): \nabla^W_X\rho_{s,\theta}\cdot D_+\varphi=-K^s\rho_\theta(X^+)\cdot\varphi+i\alpha K^s X^+\cdot\varphi+\rho_{s+1,\theta}(X^+)\cdot\varphi\\-\alpha\rho_{s,\theta}(X^+)\cdot\varphi\\
(d_s):\delta^W(\rho_{s,\theta})=0.
\end{gather*}
We will prove in the sequel that these assertions are true for all $s\in \N$. In fact, they are true for $s=1$. Indeed, it is easy to see that the r.h.s. of $(a_1)$ is equal to 
$$(2m-r_0)K+r_0\alpha=2mK+r_0(\alpha-K)=2mK-2c={\rm Scal}^W.$$
The assertion $(b_1)$ comes from the proof of Proposition \ref{thm: antoholomiddle} and $(c_1)$ comes from differentiating \eqref{eq:rhothetad+middle} and using \eqref{eq:nablax-d+middle}. Also $(d_1)$ is a consequence of \eqref{eq:diverricciform}, since the Webster scalar curvature ${\rm Scal}^W$ is constant and the torsion is zero. In the sequel, we will prove that $(a_s)\Longrightarrow (b_s), (c_s)$ and $(a_s), (d_{s-1})\Longrightarrow (d_s)$ and $(a_s), (b_s),(c_s),(d_s)\Longrightarrow (a_{s+1})$ for all $s\in \N$. For this, we consider an orthonormal frame $\{X_l, X_a\}$ of $H(M)$ such that ${\rm Ric}^W(X_a)=\mu_a X_a$ for $a\in \{1,\ldots, r\}$ and ${\rm Ric}^W(X_l)=K X_l$ for $l\in \{r+1,\ldots, 2m\}$ with $\mu_a\neq K$.  Then according to \eqref{eq:rhothetax-d+varphi}, we get that $X_a^-\cdot\varphi=0$ and $X_a^-\cdot D_+\varphi=0$. Assume now that $(a_s)$ is true. Using that $\sum_{l=r+1}^{2m} X_l\cdot JX_l\cdot\varphi=-\sum_{a=1}^{r} X_a\cdot JX_a\cdot\varphi$ as a consequence of the fact that $\varphi\in  \Sigma_{\frac{m}{2}}(H(M))$, we compute
\begin{eqnarray}\label{eq:rhostvarphimiddle}
\rho_{s,\theta}\cdot\varphi&=& \frac{1}{2}\sum_{a=1}^r\mu_a^s X_a\cdot JX_a\cdot\varphi+\frac{1}{2}\sum_{l=r+1}^{2m}K^s X_l\cdot JX_l\cdot\varphi\nonumber\\
&=&\frac{1}{2}\sum_{a=1}^r(\mu_a^s-K^s) X_a\cdot JX_a\cdot\varphi\nonumber\\
&=&-\frac{i}{2}\sum_{a=1}^r(\mu_a^s-K^s) \varphi\nonumber\\
&=&-\frac{i}{2}\left({\rm tr}(({\rm Ric}^W)^s)-2mK^s\right) \varphi\nonumber\\
&\bui{=}{\eqref{eq:tracericpower}}&-\frac{i}{2}r_0(\alpha^s-K^s)\varphi. 
\end{eqnarray}
In the same way, we show that 
\begin{equation}\label{eq:rhosd+varphimiddle}
\rho_{s,\theta}\cdot D_+\varphi=\left(-\frac{i}{2}r_0(\alpha^s-K^s)+2iK^s\right) D_+\varphi.
\end{equation}
Now, we differentiate \eqref{eq:rhostvarphimiddle} and \eqref{eq:rhosd+varphimiddle}. For any $X\in H(M)$, we compute
\begin{eqnarray*}
(\nabla^W_X\rho_{s,\theta})\cdot\varphi&=& -\frac{i}{2}r_0(\alpha^s-K^s)\nabla^W_X\varphi-\rho_{s,\theta}\cdot\nabla^W_X\varphi\\
&=&\frac{i}{2(m+2)}r_0(\alpha^s-K^s)X^-\cdot D_+\varphi+\frac{1}{m+2}\rho_{s,\theta}\cdot X^-\cdot D_+\varphi\\
&=&\frac{i}{2(m+2)}r_0(\alpha^s-K^s)X^-\cdot D_+\varphi+\frac{1}{m+2}X^-\cdot \rho_{s,\theta}\cdot D_+\varphi\\&&+\frac{2}{m+2}\rho_{s,\theta}(X^-)\cdot D_+\varphi\\
&\bui{=}{\eqref{eq:rhosd+varphimiddle}}&\frac{2iK^s}{m+2}X^-\cdot D_+\varphi-\frac{2i}{m+2}({\rm Ric}^W)^s(X^-)\cdot D_+\varphi.
\end{eqnarray*}
Now, if we take $X=X_a$ for $a\in \{1,\ldots, r\}$, we find $\nabla^W_{X_a}\rho_{s,\theta}\cdot\varphi=\frac{2i(K^s-\mu_a^s)}{m+2}X_a^-\cdot D_+\varphi=0$. Also, if we take $X=X_l$
for $l\in \{r+1,\ldots, 2m\}$, we get that $\nabla^W_{X_a}\rho_{s,\theta}\cdot\varphi=0$. Thus, we deduce that $\nabla^W_X\rho_{s,\theta}\cdot\varphi=0$ which is $(b_s)$. To show $(c_s)$, we compute 
\begin{eqnarray}\label{eq:nablarhothetad+varphimiddle}
\nabla^W_X\rho_{s,\theta}\cdot D_+\varphi&=& \left(-\frac{i}{2}r_0(\alpha^s-K^s)+2iK^s\right) \nabla^W_XD_+\varphi-\rho_{s,\theta}\cdot\nabla^W_X D_+\varphi\nonumber\\  
&\bui{=}{\eqref{eq:nablax-d+middle}}&\left(-\frac{i}{2}r_0(\alpha^s-K^s)+2iK^s\right) \left(\frac{i}{2}\rho_\theta(X^+)\cdot\varphi+\frac{\alpha}{2}X^+\cdot\varphi\right)\nonumber\\&&-\rho_{s,\theta}\cdot\left(\frac{i}{2}\rho_\theta(X^+)\cdot\varphi+\frac{\alpha}{2}X^+\cdot\varphi\right)\nonumber\\
&\bui{=}{\eqref{eq:rhostvarphimiddle}}&-K^s\rho_\theta(X^+)\cdot\varphi+i\alpha K^s X^+\cdot\varphi+\rho_{s+1,\theta}(X^+)\cdot\varphi\nonumber\\&&-\alpha\rho_{s,\theta}(X^+)\cdot\varphi\nonumber.\\
\end{eqnarray}
Assume now that both $(a_s)$ and $(d_{s-1})$ are true. From one side, we have 
\begin{equation}\label{eq:tracericcipower}
\sum_{l=1}^{2m}(\nabla^W_X{\rm Ric}^W)(e_l, ({\rm Ric}^W)^{s-1}(e_l))=\frac{1}{s}X\left({\rm tr}({\rm Ric}^W)^{s}\right)\bui{=}{\eqref{eq:tracericpower}}0.
\end{equation}
From the other side, we compute for any $X\in H(M)$
\begin{eqnarray*}
    0&\bui{=}{\eqref{eq:dwrhotheta}}&\sum_{l=1}^{2m}(d^W\rho_{\theta})(Je_l,({\rm Ric}^W)^{s-1}(e_l),JX)\\
    &=&\sum_{l=1}^{2m}(\nabla^W_{Je_l}{\rm Ric}^W)(J({\rm Ric}^W)^{s-1}(e_l),JX)\\&&+\sum_{l=1}^{2m}(\nabla^W_{({\rm Ric}^W)^{s-1}(e_l)}{\rm Ric}^W)(e_l,JX)\\&&-\sum_{l=1}^{2m}(\nabla^W_{JX}{\rm Ric}^W)(e_l,({\rm Ric}^W)^{s-1}(e_l))\\
    &\bui{=}{\eqref{eq:tracericcipower}}&2\sum_{l=1}^{2m}(\nabla^W_{e_l}{\rm Ric}^W)({\rm Ric}^{s-1}(e_l),JX)\\
&=&2g_\theta(\delta^W\rho_{s,\theta},X)-2{\rm Ric}^W(\delta^W\rho_{s-1,\theta},X)
\end{eqnarray*}
which gives $(d_s)$, since $(d_{s-1})$ is true. Now, we assume that $(a_s), (b_s), (c_s),(d_s)$ are all true, we contract Equation \eqref{eq:nablarhothetad+varphimiddle} to get
\begin{eqnarray}\label{eq:drhodeltarhomiddle} 
(d^W\rho_{s,\theta}+\delta^W\rho_{s,\theta})\cdot D_+\varphi&=&-iK^s\left(c-\frac{{\rm Scal}^W}{2}\right)\cdot\varphi-im\alpha K^s\varphi\nonumber\\
&&+\sum_{l=1}^{2m} e_l\cdot\rho_{s+1,\theta}(e_l^+)\cdot\varphi-\alpha\sum_{l=1}^{2m} e_l\cdot\rho_{s,\theta}(e_l^+)\cdot\varphi\nonumber\\
&=&2iK^s(\alpha+(m+1)K)\varphi-i{\rm tr}(({\rm Ric}^W)^{s+1})\varphi\nonumber\\&&+i\alpha{\rm tr}(({\rm Ric}^W)^{s})\varphi.
\end{eqnarray}
In the last equality, we use the following computation 
\begin{eqnarray*}
\sum_{l=1}^{2m} e_l\cdot ({\rm Ric}^W)^{s+1}(e_l^+)\cdot\varphi
=(-{\rm tr}(({\rm Ric}^W)^{s+1})+mK^{s+1})\varphi
\end{eqnarray*}
where we decompose the basis $e_l$ into $X_a$ and $X_l$ and use that $X_a^-\cdot\varphi=0$. Taking the Hermitian product of \eqref{eq:drhodeltarhomiddle} with $\varphi$ and using $(d_s)$ as well as $(b_s)$, which by contracting gives $d^W\rho_{s,\theta}\cdot\varphi=0$, allow to get after using $(a_s)$ that 
\begin{eqnarray*}
{\rm tr}(({\rm Ric}^W)^{s+1})&=&2K^s(\alpha+(m+1)K)+\alpha{\rm tr}(({\rm Ric}^W)^{s})\\
&=&2K^s(\alpha+(m+1)K)+\alpha((2m-r_0)K^s+r_0\alpha^s)\\
&=&(2m-r_0)K^{s+1}+r_0\alpha^{s+1},
\end{eqnarray*}
by the definition of $\alpha$ and $K$. This shows the required implications. Thus $(a_s)$ is true for all $s$. From Newton's relations, we deduce that the Webster Ricci tensor has two eigenvalues $\alpha$ and $K$ with multiplicity $r_0$ and $2m-r_0$.
\end{itemize}
\end{proof}

\section{Examples} \label{Examples}
In this section, we will construct examples of a CR manifold $M$ of constant Webster scalar curvature with vanishing torsion that carries an antiholomorphic CR twistor spinor $\varphi$ in $\Sigma_{r}(H(M))$ for each $r\in \{0,\ldots,m\}$. This includes the case of the middle slot $r=\frac{m}{2}$ which does not occur on compact K{\"a}hler manifolds (apart from parallel spinors; see \cite{P:11}). 


\subsection{Pseudo-Einstein CR manifold}
In this part, we will construct an example of a pseudo-Einstein CR manifold carrying an antiholomorphic CR twistor spinor. Our construction is related to the existence of a K\"ahlerian twistor spinor that arises on K\"ahler-Einstein manifolds for some spin$^c$ structures (see \cite{HM} for more details). This provides an example of Theorem \ref{thm:equalitycasegeneral} and the first part of Theorem \ref{eq:equalitycasemiddle}. \\ 

Let $(N^{2m},g_N, J)$ be a K\"ahler-Einstein compact manifold  with scalar curvature equal to $4m(m+1)$. We consider the $\mathbb{S}^1$-bundle $M:=\mathcal{P}(N)\times_{\rm det} \mathbb{S}^1$ over $N$ where $\mathcal{P}(N)$ is the unitary $U(m)$-bundle of unitary frames, as in \cite[p. 127-128]{L:18}. We denote by $\pi:M\to N$ the projection map. The Levi-Civita connection on $\mathcal{P}(N)$  induces a connection form $\zeta: TM\to i\mathbb{R}$. We set $H(M):={\rm ker}(\zeta)$ the horizontal distribution of the connection and $\hat{J}$ the lift of the complex structure $J$ on $N$. For simplicity, we also denote $\hat J$ by $J$.\\  

The curvature $d\zeta$ of this connection is the pull-back by $\pi$ of $i{\rm Ric}^N(\cdot,J\cdot)=-2i(m+1)\omega$ where $\omega=g_N(J\cdot,\cdot)$ is the K\"ahler form on $N$. By setting $\theta:=\frac{i}{m+1}\zeta$, we get that $\frac{1}{2}d\theta=\pi^*\omega$. Therefore, $(M^{2m+1},H(M),J,\theta)$ is a strictly pseudoconvex CR manifold and the induced metric $g_\theta=\frac{1}{2}d\theta(\cdot,J\cdot)=\pi^*g_N$  is just the pull-back of the metric on $N$. The vector field $T$ associated with the $1$-form $\theta$ is the fundamental vector field of the principal bundle and, thus, it is Killing. The Tanaka-Webster connection $\nabla^W$ is just the pull-back of the Levi-Civita connection of $(N,g_N)$ with vanishing torsion. Hence, the CR manifold $M$ is pseudo-Einstein with constant Webster scalar curvature ${\rm Scal}^W=4m(m+1)$.\\

According to \cite{HM}, for each $r\in\{0,\ldots,m+1\}$, there exists a spin$^c$ structure on $N$ with auxiliary line bundle $\mathcal{L}^q$ with $q=\frac{p}{m+1}(2r-m-1)\in \Z$ such that the associated spinor bundle carries a K\"ahlerian twistor spinor $\varphi_{r-1}\in \Sigma_{r-1}N$ (as well as $\varphi_{r}\in \Sigma_{r}N$), that is 
$$\nabla^N_X\varphi_{r-1}=-X^-\cdot\varphi_r,$$
for all $X\in TN$. Here $p$ is the index of $N$, that is $\mathcal{L}^p=K_N$, where $K_N$ is the canonical line bundle. By tracing this equation, we get that $\varphi_{r-1}$ is antiholomorphic ($D_-\varphi_{r-1}=0$) and that $\varphi_r=\frac{1}{2r}D_+\varphi_{r-1}$. Hence, we write that 
$$\nabla^N_X\varphi_{r-1}=-\frac{1}{2r}X^-\cdot D_+\varphi_{r-1},$$
The curvature of the auxiliary line bundle $\mathcal{L}^q$ of the spin$^c$ structure is equal to $$\frac{q}{p}i{\rm Ric}^N\circ J=2i(m+1)\frac{q}{p}\omega=2i(2r-m-1)\omega$$ 
where, as before, $\omega=g_N(J\cdot,\cdot)$ is the K\"ahler form on $N$. \\ 

The pull-back of the spin$^c$ structure on $N$ gives rise to a spin$^c$ structure on $H(M)$ with auxiliary line bundle $\widetilde{\mathcal{L}}^q:=\pi^*(\mathcal{L}^q)$. The curvature of this bundle is equal to $2i(2r-m-1)\pi^*\omega=i(2r-m-1)d\theta$. The pull-back of the spinor $\varphi_{r-1}$ into $H(M)$ gives rise to a spinor field, that we still denote by $\varphi_{r-1}$, satisfying the system of equations
\begin{equation}\label{eq:nablawspinc}
\nabla^W_T\varphi_{r-1}=0,\,\, \nabla^W_X\varphi_{r-1}=-\frac{1}{2r}X^-\cdot D_+\varphi_{r-1}
\end{equation}
for all $X\in H(M)$. Here $\nabla^W$ is the Tanaka-Webster spin$^c$ connection on $\Sigma(H(M))$.\\

Let us now denote by $\nabla$ the connection on $\widetilde{\mathcal{L}}^q$ with curvature $i(2r-m-1)d\theta$ and consider the connection $\nabla'$ on $\widetilde{\mathcal{L}}^q$ given for any section $\sigma\in\Gamma(\widetilde{\mathcal{L}}^q)$ by 
\begin{equation*}
\nabla'_X\sigma:=\left\{
\begin{matrix}
\nabla_X\sigma & \text{if}& \text{$X\in H(M)$},\\\\
\nabla_T\sigma-i(2r-m-1)\sigma & \text{if}& \text{$X=T$.}
\end{matrix}\right.
\end{equation*}
We compute the curvature $R'$ of the connection $\nabla'$. For any $X,Y\in H(M)$ parallel at some point (with respect to the Tanaka-Webster connection $\nabla^W$), we have 
\begin{eqnarray*}
R'(X,Y)\sigma&=&\nabla'_X\nabla'_Y\sigma-\nabla'_Y\nabla'_X\sigma-\nabla'_{[X,Y]}\sigma\\
&=&\nabla_X\nabla_Y\sigma-\nabla_Y\nabla_X\sigma+d\theta(X,Y)\nabla_T'\sigma\\
&=& R(X,Y)\sigma-d\theta(X,Y)\nabla_T\sigma+d\theta(X,Y)\nabla_T\sigma\\&&-i(2r-m-1)d\theta(X,Y)\sigma\\
&=&0.
\end{eqnarray*}
In the last equality, we use the fact that the curvature $R$ of the connection $\nabla$ is equal to $i(2r-m-1)d\theta$. On the other hand, we compute the curvature in the $T$-direction. For any $X\in H(M)$ parallel at some point, we write 
\begin{eqnarray*}
R'(T,X)\sigma&=&\nabla'_T\nabla'_X\sigma-\nabla'_X\nabla'_T\sigma-\nabla'_{[T,X]}\sigma\\
&=&\nabla_T\nabla_X\sigma-i(2r-m-1)\nabla_X\sigma-\nabla_X\nabla_T\sigma+i(2r-m-1)\nabla_X\sigma\\
&=&R(T,X)\sigma\\&=&i(2r-m-1)d\theta(T,X)\sigma=0,   
\end{eqnarray*}
since the torsion vanishes and $T\lrcorner d\theta=0$. Therefore, we deduce that $(\widetilde{\mathcal{L}}^q,\nabla')$ is a flat line bundle over the manifold $M$. By \cite{K}, the manifold $N$ is simply connected and by using the long Thom-Gysin sequence of the bundle $(M,\pi,N,\mathbb{S}^1)$(see \cite{M-S}), we deduce that $M$ is simply connected. Hence by \cite[Lem. 2.1]{M}, the spin$^c$ structure defined by the connection $\nabla'$ on $\widetilde{\mathcal{L}}^q$ is actually a spin structure and the corresponding connection $\nabla'$ on $\Sigma(H(M))$ is just the spin connection. Moreover, we have for all $X\in TM$ that 
\begin{equation}\label{relationspincspin}
\nabla'_X\varphi_{r-1}=\nabla^W_X\varphi_{r-1}+\frac{i}{2}\alpha(X)\varphi_{r-1},
\end{equation}
where $\alpha$ is a real $1$-form given by $i\alpha:=A'-A$, which is the difference of the linear connections of $\nabla'$ and $\nabla$ \cite{F, NR}. Now from the definition of $\nabla'$, one can easily see that $\alpha(T)=-(2r-m-1)$ and $\alpha(X)=0$ for all $X\in H(M)$. 
Therefore, plugging \eqref{eq:nablawspinc} into \eqref{relationspincspin}, we deduce that 
$$\nabla'_T\varphi_{r-1}=-\frac{i(2r-m-1)}{2}\varphi_{r-1},\,\, \nabla'_X\varphi_{r-1}= -\frac{1}{2r}X^-\cdot D_+\varphi_{r-1}$$
for all $X\in H(M)$. Hence $\varphi_{r-1}$ is an antiholomorphic CR twistor spinor for the spin structure. Finally, when $m$ is even, we take $r=\frac{m}{2}+1$ to have a CR twistor spinor in the middle slot. 

\subsection{Two eigenvalues of the Webster Ricci tensor in the middle slot}
In this part, we will construct an example of a CR manifold that carries a parallel spinor (hence a twistor spinor) and such that the Webster Ricci tensor has two different eigenvalues. This provides an example of the third part of Theorem \ref{eq:equalitycasemiddle}. \\

Let $(M_1,g_1,\xi_1)$ be a Lorentzian Einstein-Sasaki manifold of dimension $2m+1\geq 3$
with negative scalar curvature ${\rm Scal}_{g_1}=-2m(2m+1)$. The kernel $H_1$ of the dual $1$-form $\theta_1=g_1(\xi_1,\cdot)$ to the unit timelike
Reeb vector is contact and admits the complex  structure $J_1=\nabla^{g_1}\xi_1$. We assume that $M_1$ is spin and admits imaginary Killing spinors $\phi_1$ 
and $\psi_1$ with Killing number $\pm\frac{i}{2}$, i.e.,
\[\nabla_X^{g_1}\phi_1=-\frac{i}{2} X\cdot_{M_1}\phi_1\qquad\mbox{and}\qquad \nabla_X^{g_1}\psi_1=\pm\frac{i}{2} X\cdot_{M_1}\psi_1\]
for any $X\in TM_1$. The sign of the Killing number
of $\psi_1$ depends on whether $m$ is even or odd.
These spinors satisfy
\[
\xi_1\cdot_{M_1}\phi_1=-\phi_1,\qquad \xi_1\cdot_{M_1}\psi_1=-\psi_1,
\]
\[ id\theta_1\cdot_{M_1}\phi_1=2m\phi_1,\qquad id\theta_1\cdot_{M_1}\psi_1=-2m\psi_1\]
and 
\[
(X+iJ_1X)\cdot_{M_1}\phi_1=0,\qquad (X-iJ_1X)\cdot_{M_1}\psi_1=0,\qquad X\in H_1
\]
(see \cite{Baum}).\\

Similarly, let $(M_2,g_2,\xi_2)$ be a Riemannian Einstein-Sasaki spin manifold of dimension $2m+1$, but with positive
scalar curvature ${\rm Scal}_{g_2}=2m(2m+1)$. Again, let $\theta_2$ be the dual $1$-form 
with complex structure $J_2=\nabla^{g_2}\xi_2$ and $\phi_2,\psi_2$  be
Killing spinors to the real Killing numbers $\mp\frac{1}{2}$
with (see \cite[Cor. 6.10]{FK}) 
\[
\xi_2\cdot_{M_2}\phi_2=-i\phi_2,\qquad \xi_2\cdot_{M_2}\psi_2=-i\psi_2,
\]
\[ id\theta_2\cdot_{M_2}\phi_1=2m\phi_2,\qquad id\theta_2\cdot_{M_2}\psi_2=-2m\psi_2\]
and 
\[
(Y+iJ_2Y)\cdot_{M_2}\phi_2=0,\qquad(Y-iJ_2Y)\cdot_{M_2}\psi_2=0,\qquad Y\in H_2\ .
\]

Define $F:=M_1\times M_2$ with product metric 
$f:=g_1\times g_2$.
This is a Lorentzian spin manifold of dimension $2(2m+1)$.
The spinor bundle $\Sigma(F)$ of $F$
is identified with two copies of the product $\Sigma(M_1)\otimes \Sigma(M_2)$ of spinors on $M_1$ and $M_2$, i.e.,
\[ \Sigma(F)\ =\ \Sigma_+(F)\oplus \Sigma_-(F)\ \cong\ (\Sigma(M_1)\otimes \Sigma(M_2))\ \oplus\ 
(\Sigma(M_1)\otimes \Sigma(M_2))\ ,
\]
where Clifford multiplication by $X\in TM_1$ and $Y\in TM_2$ with some spinor
$\Phi^*\ =\ ( \alpha_1\otimes \alpha_2,\beta_1\otimes \beta_2 )$
is given by \cite{Baer}
\[
X\cdot \Phi^* = i\cdot (\, (X\cdot_{M_1}\beta_1)\otimes \beta_2, (-X\cdot_{M_1}\alpha_1)\otimes \alpha_2\, ),
\]
\[
Y\cdot \Phi^* = (\, \beta_1\otimes (Y\cdot_{M_2}\beta_2), \alpha_1\otimes (Y\cdot_{M_2}\alpha_2)\, )\ .
\]

We set $\Phi^*_1:=(0,\phi_1\otimes \psi_2)$ and $\Phi^*_2:=(0,\psi_1\otimes\phi_2)$.
Since $\phi_j,\psi_j,j\in\{1,2\},$ are Killing spinors, e.g., we have
$\nabla^f_X\Phi^*_1=-\frac{i}{2}(0,(X\cdot_{M_1}\phi_1)\otimes \psi_2)$ for $X\in TM_1$ and
$\nabla^f_Y\Phi^*_1=\pm\frac{1}{2}(0,\phi_1\otimes (Y\cdot_{M_2}\psi_2))$ for $Y\in TM_2$. It is straightforward
to check  that $\Phi^*_1$ and $\Phi^*_2$ are (conformal) twistor spinors on $F$, i.e.,
\[
\nabla^f_Z\Phi^*_j +\frac{1}{2(2m+1)}Z\cdot D^f\Phi^*_j=0,\qquad Z\in TF\quad (j=1,2)
\]
($D^f$ the Dirac operator).
We set $\chi:=\xi_1+\xi_2$, which is  a null Killing vector with dual $1$-form $\theta^*:=\theta_1+\theta_2$ on $F$.
It is not difficult to check that the  twistors $\Phi^*_1$ and $\Phi^*_2$ are parallel in direction of $\chi$ and satisfy $d\theta^*\cdot \Phi^*_{j}=0$, for $j=1,2$. Hence,
the Lie derivative $\mathcal{L}_\chi\Phi^*_{j}= \nabla_\chi^f\Phi^*_{j}-\frac{1}{4}d\theta^*\cdot\Phi^*_{j}=0$ vanishes as well, for all $j=1,2$. \\

We assume now that each Sasakian manifold, $M_1$ and $M_2$, arises as the bundle of units in the canonical line bundle of some
K{\"a}hler-Einstein manifold. Then $F$ is a torus bundle and the integral curves of $\chi$
are closed circles, which run on the diagonals of the  
vertical tori. The  quotient space $M:=F/\chi$ (by the flow of $\chi$)
is a smooth manifold of real dimension $4m+1$ with  projection map $\pi:F\to M$. It is straightforward to check that
$H_1\oplus H_2$ with complex structure $J_1+J_2$ projects via $\pi$ to a CR structure $(H(M),J)$ on $M$. Also
$\theta^*$ descends to some contact form $\theta$, which is an adapted pseudo-Hermitian $1$-form for $(M,H(M),J)$ of vanishing Webster torsion (since $\theta^*$ is a Killing form).
The pseudo-Hermitian Ricci form is $\rho_\theta=(m+1)(-d\theta_1+d\theta_2)$, meaning the Webster Ricci tensor has two eigenvalues $-2(m+1)$ and $2(m+1)$ of multiplicity $2m$ respectively. Hence
the Webster scalar curvature ${\rm Scal}^W=0$ vanishes.
Note that $\pi:(F,f)\to(M,H(M),J,\theta)$ is an instance of the famous
Fefferman construction (see \cite{Lee}).\\

The distribution $H(M)$ on $M$ is assumed to be spin with a spinor bundle $\Sigma(H(M))$. The  pullback of
two copies of $\Sigma(H(M))$ via $\pi$ corresponds to the spinor bundle $\Sigma(F)$. In fact, these two copies
are identified with the  annihilated spinors $Ann(\chi)$ by Clifford multiplication
with $\chi$ and with the annihilated spinors $Ann(T^*)$ of $T^*=\frac{1}{2}(-\xi_1+\xi_2)$, respectively. The twistors $\Phi^*_1$ and $\Phi^*_2$
belong to $Ann(T^*)$, are both constant in $\chi$-direction and $d\theta^*\cdot \Phi^*_{l}=0$, for $l=1,2$. Hence, $\Phi^*_1$
and $\Phi^*_2$ project to spinor fields
$\Phi_1$ and $\Phi_2$ on $M$, which are sections of
$\Sigma_m(H(M))$.\\

It remains to show that $\Phi_1$ and $\Phi_2$ are CR twistors on $(M,H(M),J)$.
For this purpose, we compare the Tanaka-Webster spinor connection $\nabla^W$ and $\nabla^f$
for basic spinors $\Phi^*$ along the Fefferman fibration $\pi$.
In general, we can prove that
\[
\pi^*\nabla^W_{\pi_*X}\Phi=\nabla^f_{X}\Phi^*+\frac{1}{2}JX\cdot T^*\cdot\Phi^*+
\frac{1}{2}\left(\tau(\pi_*X)-\frac{i}{4(m+1)}\Omega^f(\pi_*X)\cdot\chi\cdot\Phi^*\right)
\]
for $X\in H^*$,
and
\[
\pi^*\nabla^W_T\Phi= \nabla^f_{T^*}\Phi^*+\frac{i}{8(m+1)}\Omega^f\cdot\Phi^*
-\frac{i}{4(m+1)}\Omega^f(T)\cdot\chi\cdot\Phi^*,
\]
where $\Omega^f$ is the curvature of a certain connection
(in general, different from Tanaka-Webster) used in the Fefferman construction.
In the current situation, we simply have $\Omega^f=i\pi^*\rho_\theta=i(m+1)(-d\theta_1+d\theta_2)$ and $\Omega^f(T)=0$.
Then we compute, e.g. for the twistor $\Phi^*_{1}$,
\begin{eqnarray*}
\pi^*\nabla^W_{\pi_*X}\Phi_1&=&\nabla^f_{X}\Phi_1^*-\frac{1}{4}JX\cdot\chi\cdot\Phi_1^*=
-\frac{i}{2}(0,((X+iJ_1X)\cdot_{M_1}\phi_1)\otimes\psi_2)=0,\\
\pi^*\nabla^W_{\pi_*Y}\Phi_1&=&\nabla^f_{Y}\Phi_1^*+\frac{1}{4}JY\cdot\chi\cdot\Phi_1^*=
\frac{1}{2}(0,(\phi_1\otimes(Y-iJ_2Y)\cdot_{M_2}\psi_2)=0,\\
\pi^*\nabla^W_T\Phi_1&=&\nabla^f_{T^*}\Phi_1^*-\frac{1}{8(m+1)}\rho_\theta\cdot\Phi^*_1=-\frac{i(m+1)}{2}\Phi_1^*\neq 0\ ,
\end{eqnarray*}
with $X\in H_1$ and $Y\in H_2$, similarly for $\Phi_2^*$.
Thus, we have constructed $H$-parallel spinors $\Phi_1$ and $\Phi_2$ on $(M,H(M),J)$.
Note the basic holonomy algebra of $(M,H(M),J)$ is
not contained in $\mathfrak{su}(2m)$ (cf. \cite{Anton}, \cite{L:19}) and the spinors 
$\Phi_1$ and $\Phi_2$ are not basic for the underlying product of K{\"a}hler-Einstein spaces.

\end{sloppypar}

\begin{thebibliography}{9} 

\bibitem{ADS} A. Aribi, S. Dragomir, A. El Soufi,
\emph{A lower bound on the spectrum of the sublaplacian},
J. Geom. Anal. \textbf{25} (2015), no. 3, 1492–1519.

\bibitem{Baer} C. B\"ar, \emph{Extrinsic bounds for eigenvalues of the Dirac operator}, Ann. Glob. Anal. Geom. \textbf{16} (1998),  573–596.

\bibitem{Baum} H. Baum, \emph{Conformal Killing spinors and special geometric structures in Lorentzian geometry - a survey}, Proceedings of the Workshop on Special Geometric Structures in String Theory, Bonn, September 2001.  Proceedings archive of the EMS Electronic Library of Mathematics, www.univie.ac.at/EMIS/proceedings/

\bibitem{BFGK}H. Baum, Th. Friedrich, R. Grunewald, I. Kath, 
\emph{Twistor and Killing spinors on Riemannian manifolds},
Teubner-Text Nr. 124, Teubner-Verlag Stuttgart-Leipzig, 1991.

\bibitem{Bogg} A. Boggess, \emph{CR manifolds and the tangential Cauchy-Riemann complex}, Stud. Adv. Math.
CRC Press, Boca Raton, FL, 1991, xviii+364 pp.

\bibitem{CS:00} A. Čap, H. Schichl, \emph{Parabolic geometries and canonical Cartan connections}, Hokkaido Math. J. \textbf{29}  (2000), 453-505.

\bibitem{CG1} A. Čap, A.R. Gover, \emph{CR-tractors and the Fefferman space}, Indiana Univ. Math. J. \textbf{57} (2008), no. 5, 2519–2570.

\bibitem{CS} A. Čap,  J.  Slovák,  \emph{Parabolic geometries. I},
Math. Surveys Monogr., \textbf{154} American Mathematical Society, Providence, RI, 2009, x+628 pp.

\bibitem{CM} S.S. Chern, J.K. Moser, 
\emph{Real hypersurfaces in complex manifolds},
Acta Math. \textbf{133} (1974), 219–271.

\bibitem{DT:06} S. Dragomir, G. Tomassini, \emph{Differential Geometry and analysis on CR manifolds}, Progress in Mathematics, vol. 246, Birkh\"auser Boston, Inc., Boston, MA, 2006. 
\bibitem{Feff} Ch. Fefferman, \emph{Monge-Amp\`ere equations, the Bergman kernel, and geometry of pseudoconvex domains},
Ann. of Math. \textbf{103} (1976), 395-416.
\bibitem{F} Th. Friedrich, \emph{Dirac operator’s in Riemannian geometry, Graduate studies in mathematics}, Volume 25, Americain Mathematical Society. 
\bibitem{FK} Th. Friedrich, E. C. Kim, \emph{The Einstein-Dirac equations on Riemannian spin manifolds}, J. Geom. Phys. \textbf{33} (2000), 128-172.
\bibitem{G:93} P. Gauduchon, \emph{L’op\'erateur de Penrose k\"ahl\'erien et les in\'egalit\'es de Kirchberg}, preprint, 1993.
\bibitem{Anton} A. Galaev, \emph{Holonomy of K-contact sub-Riemannian manifolds}, arXiv:2502.07543 

\bibitem{Habib1} G. Habib, \emph{Eigenvalues of the transversal Dirac operator on Kähler foliations}, J. Geom. Phys. \textbf{56} (2006), no. 2, 260–270.

\bibitem{Habib2} G. Habib, \emph{Eigenvalues of the basic Dirac operator on quaternion-Kähler foliations},
Ann. Global Anal. Geom. \textbf{30} (2006), no. 3, 289–298.

\bibitem{Hij1} O. Hijazi, \emph{Eigenvalues of the Dirac operator on compact K\"ahler manifolds}, Commun. Math. Phys. \textbf{160} (1994), 563–579.

\bibitem{HM} O. Hijazi, S. Montiel, F. Urbano, \emph{Spin$^c$ geometry of K\"ahler manifolds and the Hodge Laplacian on minimal Lagrangian submanifolds}, Math. Z. \textbf{253} (2006), 821–853. 

\bibitem{JL} D. Jerison, J. Lee, \emph{The Yamabe problem on CR manifolds}, J. Diff. Geom. Phys. \textbf{25} (1987), 167-197. 


\bibitem{Kirch1} K.D. Kirchberg, \emph{
The first eigenvalue of the Dirac operator on K\"ahler man-
ifolds}, J. Geom. Phys. \textbf{4} (1990), 449–468.



\bibitem{JJK1} J.J. Kohn, \emph{
Boundaries of complex manifolds}, in Proc. Conf. Complex Analysis (Minneapolis, 1964), Springer, Berlin, 1965, 81–94.


\bibitem{JJK2} J.J. Kohn, H. Rossi, 
\emph{On the extension of holomorphic functions from the boundary of a complex manifold},
Ann. of Math. \textbf{81} (1965), 451–472.

\bibitem{K:86}K. D. Kirchberg, \emph{An estimation for the first eigenvalue of the Dirac operator on closed K\"ahler manifolds of positive scalar curvature}, Ann. Glob. Anal. Geom. \textbf{4}  (1986), 291-325.
\bibitem{K} S. Kobayashi, \emph{On compact K\"ahler manifolds with positive definite Ricci tensor}, Ann. of Math. \textbf{74} (1961), 570–574. 
\bibitem{Lee} J. M. Lee, \emph{The Fefferman metric and pseudo-Hermitian invariants}, 
Trans. Amer. Math. Soc. \textbf{296} (1986), no. 1, 411–429.
\bibitem{L:88} J. M. Lee, \emph{Pseudo-Einstein structures on CR manifolds}, Amer. J. Math. \textbf{110}  (1988), 157--178. 
\bibitem{L:18} F. Leitner, \emph{The first eigenvalue of the Kohn-Dirac operator on CR manifolds}, Diff. Geom. and Appl. \textbf{61}  (2018), 97--132. 
\bibitem{L:19} F. Leitner, \emph{Parallel spinors and basic holonomy in pseudo-Hermitian geometry}, Ann. Glob. Anal. Geom. \textbf{55}  (2019), 181--196.
\bibitem{L:21} F. Leitner, 
\emph{Invariant Dirac operators, harmonic spinors, and vanishing theorems in CR geometry},
SIGMA \textbf{17} (2021), Paper No. 011, 25 pp.
\bibitem{M-S} J. Milnor, J. Stasheff, \emph{Characteristic classes}, Princeton Univ. Press, New Jersey, 1974.
\bibitem{M:97} A. Moroianu, \emph{On Kirchberg's inequality for compact K\"ahler manifolds of even complex dimension}, Ann. Glob. Anal. Geom. \textbf{15} (1997), 235-242. 
\bibitem{M} A. Moroianu, \emph{Parallel and Killing spinors on spin$^c$ manifolds}, Comm. Math. Phys. \textbf{187} (1997), 417–427.
\bibitem{NR} R. Nakad, J. Roth, \emph{The Spin$^c$ Dirac operator on hypersurfaces and applications},  Diff. Geom. and Appl. \textbf{31}  (2013), 93-103.
\bibitem{Pe:05} R. Petit, \emph{Spin$^c$-structures and Dirac operators on contact manifolds}, Diff. Geom. and Appl. \textbf{22} (2005), 229-252. 
\bibitem{P:11} M. Pilca, \emph{Kählerian twistor spinors}, Math. Z. \textbf{268}  (2011), 223-255. 
\bibitem{S:17} Ch. Stadtm\"uller, \emph{Horizontal Dirac operators in CR geometry}, PhD thesis, Humboldt University, Berlin, 2017, https://edoc.hu-berlin.de/handle/18452/18801.
%
\bibitem{Tanaka} N. Tanaka,
\emph{A differential geometric study on strongly pseudoconvex manifolds}, Lectures in Mathematics,
Vol. \textbf{9}, Kinokuniya Book-Store Co., Ltd., Tokyo, 1975
\bibitem{Tondeur} Ph. Tondeur,
\emph{Foliations on Riemannian manifolds}, Springer-Verlag, New York-Heidelberg-LondonParis-Tokyo, 1988
\bibitem{Web} S.M. Webster, 
\emph{Pseudo-Hermitian structures on a real hypersurface}, 
J. Diff. Geom. \textbf{13} (1978), 25–41.











\end{thebibliography}
\end{document}